\documentclass[a4paper,10pt]{article}
\usepackage{pdfsync}

\usepackage{amscd,amssymb,amsthm,amsmath,graphicx,verbatim,mathrsfs}
\usepackage[dvips]{hyperref}
\usepackage{textcomp}
\usepackage[OT1,T1]{fontenc}

\usepackage[dvips,ps,all,graph,curve]{xypic}
\makeatletter
 \def\setlabelmargin#1{\labelmargin@=#1\relax }
\makeatother

\newcommand{\ot}{{\otimes}}

\newcommand{\cZ}{{\cal Z}}
\newcommand{\cC}{{\cal C}}
\newcommand{\caD}{{\cal D}}
\newcommand{\C}{{\bf C}}

\newcommand{\g}{{\frak g}}
\newcommand{\Z}{{\bf Z}}
\newcommand{\id}{\operatorname{id}}
\newcommand{\idI}{\operatorname{I}}
\newcommand{\idIs}{\textrm{\tiny{$\idI$}}}
\newcommand{\modx}{\operatorname{X}}
\newcommand{\modxs}{\textrm{\tiny{$\modx$}}}
\newcommand{\modm}{\operatorname{m}}
\newcommand{\modms}{\textrm{\tiny{$\modm$}}}
\newcommand{\modn}{\operatorname{n}}
\newcommand{\modns}{\textrm{\tiny{$\modn$}}}

\newcommand{\bm}{{\cal{M}}}
\newcommand{\bc}{{\cal{C}}}

\newcommand{\V}{{\cal{V}}}
\newcommand{\U}{{\cal{U}}}
\newcommand{\W}{{\cal{W}}}
\newcommand{\M}{{\cal{M}}}
\newcommand{\fibfrmod}{\operatorname{Fib}}
\newcommand{\fibfr}{\textrm{$\fibfrmod$}}
\newcommand{\fib}{\mathcal{F}ib}
\newcommand{\Vir}{{\cal V}{\it ir}}
\newcommand{\Vect}{{\cal V}{\it ect}}

\newcommand{\Rep}{{\cal R}{\it ep}}

\newcommand{\ev}{\cal{V}}

\newcommand{\mmapcs}[4]{\ar@/#2/@<#4ex>[#1]|-*=0{\rotatebox{#3}{\tiny |}}} 

\newcommand{\sumo}{{\sum'}}
\DeclareMathOperator{\osum}{\sumo}

\newtheorem{theorem}{Theorem}
\newtheorem{lemma}[theorem]{Lemma}
\newtheorem{proposition}[theorem]{Proposition}
\newtheorem{corollary}[theorem]{Corollary}

\newenvironment{example}[1][Example]{\begin{trivlist}
\item[\hskip \labelsep {\bfseries #1}]}{\end{trivlist}}
\newenvironment{remark}[1][Remark]{\begin{trivlist}
\item[\hskip \labelsep {\bfseries #1}]}{\end{trivlist}}

\begin{document}

\newbox\Treetwo
\setbox\Treetwo =\hbox{\xygraph{ !{0;/r1.0pc/:;/u1.0pc/::}[]
   -[u]
  ( -[ul]
    ,-[ur]
  )}}
\newcommand*{\TreeTwo}{\copy\Treetwo}

\newbox\Treetwoid
\setbox\Treetwoid =\hbox{\xygraph{ !{0;/r1.0pc/:;/u1.0pc/::}[]
   -[u]_*!LD{\idIs}
  ( -[ul]
    ,-[ur]
  )}}
\newcommand*{\TreeTwoid}{\copy\Treetwoid}

\newbox\TreethreeA
\setbox \TreethreeA =\hbox{\xygraph{ !{0;/r1.0pc/:;/u1.0pc/::}[]
   -[u]
  ( -[ul]
  ( -[ul]
  )
    ,-[ur]
  ( -[ur]
  ,-[ul]
  )
  )}}
\newcommand*{\TreeThreeA}{\copy \TreethreeA}

\newbox\TreethreeB
\setbox \TreethreeB =\hbox{\xygraph{ !{0;/r1.0pc/:;/u1.0pc/::}[]
   -[u]
  ( -[ul]
  ( -[ul]
  )
    ,-[ur]_*!LU{\idIs}
  ( -[ur]
  ,-[ul]
  )
  )}}
\newcommand*{\TreeThreeB}{\copy \TreethreeB}

\newbox\TreethreeArev
\setbox \TreethreeArev =\hbox{\xygraph{ !{0;/r1.0pc/:;/u1.0pc/::}[]
   -[u]
  ( -[ul]
  ( -[ul]
  ,-[ur]
  )
    ,-[ur]
  ( -[ur]
  )
  )}}
\newcommand*{\TreeThreeArev}{\copy \TreethreeArev}

\newbox\TreethreeBrev
\setbox \TreethreeBrev =\hbox{\xygraph{ !{0;/r1.0pc/:;/u1.0pc/::}[]
   -[u]
  ( -[ul]^*!RU{\idIs}
  ( -[ul]
  ,-[ur]
  )
    ,-[ur]
  ( -[ur]
  )
  )}}
\newcommand*{\TreeThreeBrev}{\copy \TreethreeBrev}

\newbox\Treethreeid
\setbox \Treethreeid =\hbox{\xygraph{ !{0;/r1.0pc/:;/u1.0pc/::}[]
     -[u]_*!LD{\idIs}
  ( -[ul]
  ( -[ul]
  )
    ,-[ur]
  ( -[ur]
  ,-[ul]
  )
  )}}
\newcommand*{\TreeThreeid}{\copy \Treethreeid}

\newbox\Treethreeidrev
\setbox \Treethreeidrev =\hbox{\xygraph{ !{0;/r1.0pc/:;/u1.0pc/::}[]
     -[u]_*!LD{\idIs}
 (-[ul]
 (-[ul]
 ,-[ur]
 )
 ,-[ur]
 (
 -[ur]
 )
 )}}
\newcommand*{\TreeThreeidrev}{\copy \Treethreeidrev}

\newbox\TreethreeAs
\setbox \TreethreeAs =\hbox{\xygraph{ !{0;/r0.5pc/:;/u0.5pc/::}[]
   -[u]
  ( -[ul]
  ( -[ul]
  )
    ,-[ur]
  ( -[ur]
  ,-[ul]
  )
  )}}
\newcommand*{\TreeThreeAs}{\copy \TreethreeAs}

\newbox\TreethreeBs
\setbox \TreethreeBs =\hbox{\xygraph{ !{0;/r0.5pc/:;/u0.5pc/::}[]
   -[u]
  ( -[ul]
  ( -[ul]
  )
    ,-[ur]_*!LU{\idIs}
  ( -[ur]
  ,-[ul]
  )
  )}}
\newcommand*{\TreeThreeBs}{\copy \TreethreeBs}

\newbox\TreethreeArevs
\setbox \TreethreeArevs =\hbox{\xygraph{ !{0;/r0.5pc/:;/u0.5pc/::}[]
   -[u]
  ( -[ul]
  ( -[ul]
  ,-[ur]
  )
    ,-[ur]
  ( -[ur]
  )
  )}}
\newcommand*{\TreeThreeArevs}{\copy \TreethreeArevs}

\newbox\TreethreeBrevs
\setbox \TreethreeBrevs =\hbox{\xygraph{ !{0;/r0.5pc/:;/u0.5pc/::}[]
   -[u]
  ( -[ul]^*!RU{\idIs}
  ( -[ul]
  ,-[ur]
  )
    ,-[ur]
  ( -[ur]
  )
  )}}
\newcommand*{\TreeThreeBrevs}{\copy \TreethreeBrevs}

\newbox\Treethreeids
\setbox \Treethreeids =\hbox{\xygraph{ !{0;/r0.5pc/:;/u0.5pc/::}[]
     -[u]_*!LD{\idIs}
  ( -[ul]
  ( -[ul]
  )
    ,-[ur]
  ( -[ur]
  ,-[ul]
  )
  )}}
\newcommand*{\TreeThreeids}{\copy \Treethreeids}

\newbox\Treethreeidrevs
\setbox \Treethreeidrevs =\hbox{\xygraph{ !{0;/r0.5pc/:;/u0.5pc/::}[]
     -[u]_*!LD{\idIs}
 (-[ul]
 (-[ul]
 ,-[ur]
 )
 ,-[ur]
 (
 -[ur]
 )
 )}}
\newcommand*{\TreeThreeidrevs}{\copy \Treethreeidrevs}


\newbox\TreefouridAtopa
\setbox \TreefouridAtopa =\hbox{\setlabelmargin{2pt}
\xygraph{ !{0;/r0.5pc/:;/u0.5pc/::}[]
     -[u]_*+!DL(.5){\idIs}
(-[ul]
(-[ul] 
(-[ul]
)
)
,-[ur]
(-[ur]
(-[ul]
)
(-[ur]
)
,-[ul]
(-[ul]
)
)
)
}}
\newcommand*{\TreeFouridAtopa}{\copy \TreefouridAtopa}

\newbox\TreefouridAtopaid
\setbox \TreefouridAtopaid =\hbox{\setlabelmargin{2pt}
\xygraph{ !{0;/r0.5pc/:;/u0.5pc/::}[]
     -[u]_*+!DL(.5){\idIs}
(-[ul]
(-[ul] 
(-[ul]
)
)
,-[ur]
(-[ur]_*+!UL(0.9){\idIs}
(-[ul]
)
(-[ur]
)
,-[ul]
(-[ul]
)
)
)
}}
\newcommand*{\TreeFouridAtopaid}{\copy \TreefouridAtopaid}

\newbox\TreefouridAtopb
\setbox \TreefouridAtopb =\hbox{\setlabelmargin{2pt}
\xygraph{ !{0;/r0.5pc/:;/u0.5pc/::}[]
     -[u]_*+!DL(.5){\idIs}
(-[ul]
(-[ul]
(-[ul]
,-[ur]
)
)
,-[ur]
(-[ur]
(-[ur]
,-[ul]
)
)
)
}}
\newcommand*{\TreeFouridAtopb}{\copy \TreefouridAtopb}

\newbox\TreefouridAtopbid
\setbox \TreefouridAtopbid =\hbox{\setlabelmargin{2pt}
\xygraph{ !{0;/r0.5pc/:;/u0.5pc/::}[]
     -[u]_*+!DL(.5){\idIs}
(-[ul]
(-[ul]^*+!UR(.9){\idIs}
(-[ul]
,-[ur]
)
)
,-[ur]
(-[ur]_*+!UL(.9){\idIs}
(-[ur]
,-[ul]
)
)
)
}}
\newcommand*{\TreeFouridAtopbid}{\copy \TreefouridAtopbid}

\newbox\TreefouridAtopc
\setbox \TreefouridAtopc =\hbox{\setlabelmargin{2pt}
\xygraph{ !{0;/r0.5pc/:;/u0.5pc/::}[]
     -[u]_*+!DL(.5){\idIs}
(-[ul]
(-[ul]
(-[ur]
,-[ul]
)
,-[ur]
(-[ur]
)
)
,-[ur]
(-[ur]
(-[ur]
)
)
)  
}}
\newcommand*{\TreeFouridAtopc}{\copy \TreefouridAtopc}

\newbox\TreefouridAtopcid
\setbox \TreefouridAtopcid =\hbox{\setlabelmargin{2pt}
\xygraph{ !{0;/r0.5pc/:;/u0.5pc/::}[]
     -[u]_*+!DL(.5){\idIs}
(-[ul]
(-[ul]^*!RU{\idIs}
(-[ur]
,-[ul]
)
,-[ur]
(-[ur]
)
)
,-[ur]
(-[ur]
(-[ur]
)
)
)  
}}
\newcommand*{\TreeFouridAtopcid}{\copy \TreefouridAtopcid}

\newbox\TreefouridAbota
\setbox \TreefouridAbota =\hbox{\setlabelmargin{2pt}
\xygraph{ !{0;/r0.5pc/:;/u0.5pc/::}[]
     -[u]_*+!DL(.5){\idIs}
(-[ul]
(-[ul]
(-[ul]
)
)
,-[ur]
(-[ur]
(-[ur]
)
,-[ul]
(-[ul]
,-[ur]
)
)
) 
}}
\newcommand*{\TreeFouridAbota}{\copy \TreefouridAbota}

\newbox\TreefouridAbotaid
\setbox \TreefouridAbotaid =\hbox{\setlabelmargin{2pt}
\xygraph{ !{0;/r0.5pc/:;/u0.5pc/::}[]
     -[u]_*+!DL(.5){\idIs}
(-[ul]
(-[ul]
(-[ul]
)
)
,-[ur]
(-[ur]
(-[ur]
)
,-[ul]^*!RU{\idIs}
(-[ul]
,-[ur]
)
)
) 
}}
\newcommand*{\TreeFouridAbotaid}{\copy \TreefouridAbotaid}

\newbox\TreefouridAbotb
\setbox \TreefouridAbotb =\hbox{\setlabelmargin{2pt}
\xygraph{ !{0;/r0.5pc/:;/u0.5pc/::}[]
     -[u]_*+!DL(.5){\idIs}
(-[ul]
(-[ul]
(-[ul]
)
,-[ur]
(-[ur]
,-[ul]
)
)
,-[ur]
(-[ur]
(-[ur]
)
)
)
}}
\newcommand*{\TreeFouridAbotb}{\copy \TreefouridAbotb}

\newbox\TreefouridAbotbid
\setbox \TreefouridAbotbid =\hbox{\setlabelmargin{2pt}
 \xygraph{ !{0;/r0.5pc/:;/u0.5pc/::}[]
     -[u]_*+!DL(.5){\idIs}
(-[ul]
(-[ul]
(-[ul]
)
,-[ur]_*!LU{\idIs}
(-[ur]
,-[ul]
)
)
,-[ur]
(-[ur]
(-[ur]
)
)
)
}}
\newcommand*{\TreeFouridAbotbid}{\copy \TreefouridAbotbid}

\newbox\TreefouridBtopa
\setbox \TreefouridBtopa =\hbox{\setlabelmargin{2pt}
\xygraph{ !{0;/r0.5pc/:;/u0.5pc/::}[]
     -[u]
(-[ul]
(-[ul] 
(-[ul]
)
)
,-[ur]
(-[ur]
(-[ul]
)
(-[ur]
)
,-[ul]
(-[ul]
)
)
)
}}
\newcommand*{\TreeFouridBtopa}{\copy \TreefouridBtopa}

\newbox\TreefouridBtopaid
\setbox \TreefouridBtopaid =\hbox{\setlabelmargin{2pt}
\xygraph{ !{0;/r0.5pc/:;/u0.5pc/::}[]
     -[u]
(-[ul]
(-[ul] 
(-[ul]
)
)
,-[ur]
(-[ur]_*!LU{\idIs}
(-[ul]
)
(-[ur]
)
,-[ul]
(-[ul]
)
)
)
}}
\newcommand*{\TreeFouridBtopaid}{\copy \TreefouridBtopaid}

\newbox\TreefouridBtopaidid
\setbox \TreefouridBtopaidid =\hbox{\setlabelmargin{2pt}
\xygraph{ !{0;/r0.5pc/:;/u0.5pc/::}[]
     -[u]
(-[ul]
(-[ul] 
(-[ul]
)
)
,-[ur]_*!LU{\idIs}
(-[ur]
(-[ul]
)
(-[ur]
)
,-[ul]
(-[ul]
)
)
)
}}
\newcommand*{\TreeFouridBtopaidid}{\copy \TreefouridBtopaidid}

\newbox\TreefouridBbota
\setbox \TreefouridBbota =\hbox{\setlabelmargin{2pt}
\xygraph{ !{0;/r0.5pc/:;/u0.5pc/::}[]
     -[u]
(-[ul]
(-[ul]
(-[ul]
)
)
,-[ur]
(-[ur]
(-[ur]
)
,-[ul]
(-[ul]
,-[ur]
)
)
) 
}}
\newcommand*{\TreeFouridBbota}{\copy \TreefouridBbota}

\newbox\TreefouridBbotaid
\setbox \TreefouridBbotaid =\hbox{\setlabelmargin{2pt}
\xygraph{ !{0;/r0.5pc/:;/u0.5pc/::}[]
     -[u]
(-[ul]
(-[ul]
(-[ul]
)
)
,-[ur]
(-[ur]
(-[ur]
)
,-[ul]^*!RU{\idIs}
(-[ul]
,-[ur]
)
)
) 
}}
\newcommand*{\TreeFouridBbotaid}{\copy \TreefouridBbotaid}

\newbox\TreefouridBtopb
\setbox \TreefouridBtopb =\hbox{\setlabelmargin{2pt}
\xygraph{ !{0;/r0.5pc/:;/u0.5pc/::}[]
     -[u]
(-[ul]
(-[ul]
(-[ul]
,-[ur]
)
)
,-[ur]
(-[ur]
(-[ur]
,-[ul]
)
)
)
}}
\newcommand*{\TreeFouridBtopb}{\copy \TreefouridBtopb}

\newbox\TreefouridBtopbid
\setbox \TreefouridBtopbid =\hbox{\setlabelmargin{2pt}
\xygraph{ !{0;/r0.5pc/:;/u0.5pc/::}[]
     -[u]
(-[ul]
(-[ul]^*+!UR(.9){\idIs}
(-[ul]
,-[ur]
)
)
,-[ur]
(-[ur]
(-[ur]
,-[ul]
)
)
)
}}
\newcommand*{\TreeFouridBtopbid}{\copy \TreefouridBtopbid}

\newbox\TreefouridBtopbidid
\setbox \TreefouridBtopbidid =\hbox{\setlabelmargin{2pt}
\xygraph{ !{0;/r0.5pc/:;/u0.5pc/::}[]
     -[u]
(-[ul]
(-[ul]
(-[ul]
,-[ur]
)
)
,-[ur]
(-[ur]_*+!UL(.9){\idIs}
(-[ur]
,-[ul]
)
)
)
}}
\newcommand*{\TreeFouridBtopbidid}{\copy \TreefouridBtopbidid}

\newbox\TreefouridBbotb
\setbox \TreefouridBbotb =\hbox{\setlabelmargin{2pt}
\xygraph{ !{0;/r0.5pc/:;/u0.5pc/::}[]
     -[u]
(-[ul]
(-[ul]
(-[ul]
)
,-[ur]
(-[ur]
,-[ul]
)
)
,-[ur]
(-[ur]
(-[ur]
)
)
)
}}
\newcommand*{\TreeFouridBbotb}{\copy \TreefouridBbotb}

\newbox\TreefouridBbotbid
\setbox \TreefouridBbotbid =\hbox{\setlabelmargin{2pt}
 \xygraph{ !{0;/r0.5pc/:;/u0.5pc/::}[]
     -[u]
(-[ul]
(-[ul]
(-[ul]
)
,-[ur]_*!LU{\idIs}
(-[ur]
,-[ul]
)
)
,-[ur]
(-[ur]
(-[ur]
)
)
)
}}
\newcommand*{\TreeFouridBbotbid}{\copy \TreefouridBbotbid}

\newbox\TreefouridBbotbidid
\setbox \TreefouridBbotbidid =\hbox{\setlabelmargin{2pt}
 \xygraph{ !{0;/r0.5pc/:;/u0.5pc/::}[]
     -[u]
(-[ul]^*!UR{\idIs}
(-[ul]
(-[ul]
)
,-[ur]
(-[ur]
,-[ul]
)
)
,-[ur]
(-[ur]
(-[ur]
)
)
)
}}
\newcommand*{\TreeFouridBbotbidid}{\copy \TreefouridBbotbidid}

\newbox\TreefouridBtopc
\setbox \TreefouridBtopc =\hbox{\setlabelmargin{2pt}
\xygraph{ !{0;/r0.5pc/:;/u0.5pc/::}[]
     -[u]
(-[ul]
(-[ul]
(-[ur]
,-[ul]
)
,-[ur]
(-[ur]
)
)
,-[ur]
(-[ur]
(-[ur]
)
)
)  
}}
\newcommand*{\TreeFouridBtopc}{\copy \TreefouridBtopc}

\newbox\TreefouridBtopcid
\setbox \TreefouridBtopcid =\hbox{\setlabelmargin{2pt}
\xygraph{ !{0;/r0.5pc/:;/u0.5pc/::}[]
     -[u]
(-[ul]
(-[ul]^*!RU{\idIs}
(-[ur]
,-[ul]
)
,-[ur]
(-[ur]
)
)
,-[ur]
(-[ur]
(-[ur]
)
)
)  
}}
\newcommand*{\TreeFouridBtopcid}{\copy \TreefouridBtopcid}

\newbox\TreefouridBtopcidb
\setbox \TreefouridBtopcidb =\hbox{\setlabelmargin{2pt}
\xygraph{ !{0;/r0.5pc/:;/u0.5pc/::}[]
     -[u]
(-[ul]^*!RU{\idIs}
(-[ul]
(-[ur]
,-[ul]
)
,-[ur]
(-[ur]
)
)
,-[ur]
(-[ur]
(-[ur]
)
)
)  
}}
\newcommand*{\TreeFouridBtopcidb}{\copy \TreefouridBtopcidb}

\newbox\modactmn
\setbox \modactmn =\hbox{\setlabelmargin{2pt}
\xygraph{ !{0;/r1pc/:;/u1pc/::}[]
     -[u]_<*!DL(.5){\modns}
(-[ur]_>*!UL(.5){\modms}
)
(-[ul]^>*!RU(.5){\modxs}
)
}}
\newcommand*{\Modactmn}{\copy \modactmn}

\newbox\modactnn
\setbox \modactnn =\hbox{\setlabelmargin{2pt}
\xygraph{ !{0;/r1pc/:;/u1pc/::}[]
     -[u]_<*!DL(.5){\modns}
(-[ur]_>*!UL(.5){\modns}
)
(-[ul]^>*!RU(.5){\modxs}
)
}}
\newcommand*{\Modactnn}{\copy \modactnn}

\newbox\modactnm
\setbox \modactnm =\hbox{\setlabelmargin{2pt}
\xygraph{ !{0;/r1pc/:;/u1pc/::}[]
     -[u]_<*!DL(.5){\modms}
(-[ur]_>*!UL(.5){\modns}
)
(-[ul]^>*!RU(.5){\modxs}
)
}}
\newcommand*{\Modactnm}{\copy \modactnm}

\newbox\modassocmnm
\setbox \modassocmnm =\hbox{\xygraph{ !{0;/r1.0pc/:;/u1.0pc/::}[]
   -[u]_<*!DL(.5){\modms}
(-[ul]
(-[ul]^>*!RU{\modxs}
)
)
(-[ur]_*!LU(1){\modns}
(-[ul]^>*!RU{\modxs}
)
(-[ur]_>*!LU(1){\modms}
)
)
}}
\newcommand*{\Modassocmnm}{\copy \modassocmnm}

\newbox\modassocmmrev
\setbox \modassocmmrev =\hbox{\xygraph{ !{0;/r1.0pc/:;/u1.0pc/::}[]
   -[u]_<*!DL(.5){\modms}
(-[ul]^*!RU{\idIs}
(-[ul]^>*!RU{\modxs}
,-[ur]_>*!LU{\modxs}
)
)
(-[ur]
(-[ur]_>*!LU(1){\modms}
)
)
}}
\newcommand*{\Modassocmmrev}{\copy \modassocmmrev}

\newbox\modassocmnn
\setbox \modassocmnn =\hbox{\xygraph{ !{0;/r1.0pc/:;/u1.0pc/::}[]
   -[u]_<*!DL(.5){\modns}
(-[ul]
(-[ul]^>*!RU{\modxs}
)
)
(-[ur]_*!LU(1){\modns}
(-[ul]^>*!RU{\modxs}
)
(-[ur]_>*!LU(1){\modms}
)
)
}}
\newcommand*{\Modassocmnn}{\copy \modassocmnn}

\newbox\modassocmnrev
\setbox \modassocmnrev =\hbox{\xygraph{ !{0;/r1.0pc/:;/u1.0pc/::}[]
   -[u]_<*!DL(.5){\modns}
(-[ul]^*!RU{\modxs}
(-[ul]^>*!RU{\modxs}
,-[ur]_>*!LU{\modxs}
)
)
(-[ur]
(-[ur]_>*!LU(1){\modms}
)
)
}}
\newcommand*{\Modassocmnrev}{\copy \modassocmnrev}

\newbox\modassocnnm
\setbox \modassocnnm =\hbox{\xygraph{ !{0;/r1.0pc/:;/u1.0pc/::}[]
   -[u]_<*!DL(.5){\modms}
(-[ul]
(-[ul]^>*!RU{\modxs}
)
)
(-[ur]_*!LU(1){\modns}
(-[ul]^>*!RU{\modxs}
)
(-[ur]_>*!LU(1){\modns}
)
)
}}
\newcommand*{\Modassocnnm}{\copy \modassocnnm}

\newbox\modassocnmrev
\setbox \modassocnmrev =\hbox{\xygraph{ !{0;/r1.0pc/:;/u1.0pc/::}[]
   -[u]_<*!DL(.5){\modms}
(-[ul]^*!RU{\modxs}
(-[ul]^>*!RU{\modxs}
,-[ur]_>*!LU{\modxs}
)
)
(-[ur]
(-[ur]_>*!LU(1){\modns}
)
)
}}
\newcommand*{\Modassocnmrev}{\copy \modassocnmrev}

\newbox\modassocnnn
\setbox \modassocnnn =\hbox{\xygraph{ !{0;/r1.0pc/:;/u1.0pc/::}[]
   -[u]_<*!DL(.5){\modns}
(-[ul]
(-[ul]^>*!RU{\modxs}
)
)
(-[ur]_*!LU(1){\modns}
(-[ul]^>*!RU{\modxs}
)
(-[ur]_>*!LU(1){\modns}
)
)
}}
\newcommand*{\Modassocnnn}{\copy \modassocnnn}

\newbox\modassocnmn
\setbox \modassocnmn =\hbox{\xygraph{ !{0;/r1.0pc/:;/u1.0pc/::}[]
   -[u]_<*!DL(.5){\modns}
(-[ul]
(-[ul]^>*!RU{\modxs}
)
)
(-[ur]_*!LU(1){\modms}
(-[ul]^>*!RU{\modxs}
)
(-[ur]_>*!LU(1){\modns}
)
)
}}
\newcommand*{\Modassocnmn}{\copy \modassocnmn}

\newbox\modassocnnrev
\setbox \modassocnnrev =\hbox{\xygraph{ !{0;/r1.0pc/:;/u1.0pc/::}[]
   -[u]_<*!DL(.5){\modns}
(-[ul]^*!RU{\modxs}
(-[ul]^>*!RU{\modxs}
,-[ur]_>*!LU{\modxs}
)
)
(-[ur]
(-[ur]_>*!LU(1){\modns}
)
)
}}
\newcommand*{\Modassocnnrev}{\copy \modassocnnrev}

\newbox\modassocnnidrev
\setbox \modassocnnidrev =\hbox{\xygraph{ !{0;/r1.0pc/:;/u1.0pc/::}[]
   -[u]_<*!DL(.5){\modns}
(-[ul]^*!RU{\idIs}
(-[ul]^>*!RU{\modxs}
,-[ur]_>*!LU{\modxs}
)
)
(-[ur]
(-[ur]_>*!LU(1){\modns}
)
)
}}
\newcommand*{\Modassocnnidrev}{\copy \modassocnnidrev}


\newbox\modpentonea
\setbox \modpentonea =\hbox{\xygraph{ !{0;/r0.8pc/:;/u0.8pc/::}[]
   -[u]_<*!DL(.5){\modms}
(-[ul]
(-[ul]
(-[ul]^>*!RU{\modxs}
)
)
)
(-[ur]_>*!LD(3.5){\modns}
(-[ul]
(-[ul]^>*!RU{\modxs}
)
)
(-[ur]_>*!LD(3.5){\modns}
(-[ul]^>*!RU{\modxs}
,-[ur]_>*!LU{\modms}
)
)
)
}}
\newcommand*{\Modpentonea}{\copy \modpentonea}

\newbox\modpentoneb
\setbox \modpentoneb =\hbox{\setlabelmargin{2pt}
\xygraph{ !{0;/r0.8pc/:;/u0.8pc/::}[]
     -[u]_<*+!DL(.5){\modms}
(-[ul]^>*!LD(-0.3){\modxs}
(-[ul]
(-[ul]^>*!RU(0.5){\modxs}
,-[ur]_>*!LU(0.5){\modxs}
)
)
,-[ur]_>*!RU(-0.1){\modns}
(-[ur]
(-[ur]_>*!LU(0.3){\modms}
,-[ul]^>*!RU(0.5){\modxs}
)
)
)
}}
\newcommand*{\Modpentoneb}{\copy \modpentoneb}

\newbox\modpentonec
\setbox \modpentonec =\hbox{\setlabelmargin{2pt}
\xygraph{ !{0;/r0.8pc/:;/u0.8pc/::}[]
     -[u]_<*+!DL(.5){\modms}
(-[ul]^-*!LD(-0.3){\idIs}
(-[ul]^-*!LD(-0.3){\modxs}
(-[ur]_>*!LU(0.5){\modxs}
,-[ul]^>*!RU(0.5){\modxs}
)
,-[ur]
(-[ur]_>*!LU(0.5){\modxs}
)
)
,-[ur]
(-[ur]
(-[ur]_>*!LU(0.3){\modms}
)
)
)  
}}
\newcommand*{\Modpentonec}{\copy \modpentonec}

\newbox\modpentoned
\setbox \modpentoned =\hbox{\setlabelmargin{2pt}
\xygraph{ !{0;/r0.8pc/:;/u0.8pc/::}[]
     -[u]_<*+!DL(.5){\modms}
(-[ul]
(-[ul]
(-[ul]^>*!RU(0.5){\modxs}
)
)
,-[ur]_*!LD(-0.4){\modns}
(-[ur]
(-[ur]_>*!LU(0.3){\modms}
)
,-[ul]^-*!LD(-0.3){\modxs}
(-[ul]^>*!RU(0.5){\modxs}
,-[ur]_>*!LU(0.5){\modxs}
)
)
) 
}}
\newcommand*{\Modpentoned}{\copy \modpentoned}

\newbox\modpentonee
\setbox \modpentonee =\hbox{\setlabelmargin{2pt}
\xygraph{ !{0;/r0.8pc/:;/u0.8pc/::}[]
     -[u]_<*+!DL(.5){\modms}
(-[ul]^-*!LD(-0.3){\idIs}
(-[ul]
(-[ul]^>*!RU(0.5){\modxs}
)
,-[ur]_>*!LD(2.5){\modxs}
(-[ur]_>*!LU(0.5){\modxs}
,-[ul]^>*!RU(0.5){\modxs}
)
)
,-[ur]
(-[ur]
(-[ur]_>*!LU(0.5){\modms}
)
)
)
}}
\newcommand*{\Modpentonee}{\copy \modpentonee}


\newbox\modpenttwoa
\setbox \modpenttwoa =\hbox{\xygraph{ !{0;/r0.8pc/:;/u0.8pc/::}[]
   -[u]_<*!DL(.5){\modns}
(-[ul]
(-[ul]
(-[ul]^>*!RU{\modxs}
)
)
)
(-[ur]_>*!LD(3.5){\modns}
(-[ul]
(-[ul]^>*!RU{\modxs}
)
)
(-[ur]_>*!LD(3.5){\modns}
(-[ul]^>*!RU{\modxs}
,-[ur]_>*!LU{\modms}
)
)
)
}}
\newcommand*{\Modpenttwoa}{\copy \modpenttwoa}

\newbox \modpenttwob
\setbox \modpenttwob =\hbox{\setlabelmargin{2pt}
\xygraph{ !{0;/r0.8pc/:;/u0.8pc/::}[]
     -[u]_<*+!DL(.5){\modns}
(-[ul]^>*!LD(-0.3){\modxs}
(-[ul]
(-[ul]^>*!RU(0.5){\modxs}
,-[ur]_>*!LU(0.5){\modxs}
)
)
,-[ur]_>*!RU(-0.1){\modns}
(-[ur]
(-[ur]_>*!LU(0.3){\modms}
,-[ul]^>*!RU(0.5){\modxs}
)
)
)
}}
\newcommand*{\Modpenttwob}{\copy \modpenttwob}

\newbox \modpenttwobb
\setbox \modpenttwobb =\hbox{\setlabelmargin{2pt}
\xygraph{ !{0;/r0.8pc/:;/u0.8pc/::}[]
     -[u]_<*+!DL(.5){\modns}
(-[ul]^>*!LD(-0.3){\idIs}
(-[ul]
(-[ul]^>*!RU(0.5){\modxs}
,-[ur]_>*!LU(0.5){\modxs}
)
)
,-[ur]_>*!RU(-0.1){\modns}
(-[ur]
(-[ur]_>*!LU(0.3){\modms}
,-[ul]^>*!RU(0.5){\modxs}
)
)
)
}}
\newcommand*{\Modpenttwobb}{\copy \modpenttwobb}

\newbox \modpenttwoc
\setbox \modpenttwoc =\hbox{\setlabelmargin{2pt}
\xygraph{ !{0;/r0.8pc/:;/u0.8pc/::}[]
     -[u]_<*+!DL(.5){\modns}
(-[ul]^-*!LD(-0.3){\modxs}
(-[ul]^-*!LD(-0.3){\modxs}
(-[ur]_>*!LU(0.5){\modxs}
,-[ul]^>*!RU(0.5){\modxs}
)
,-[ur]
(-[ur]_>*!LU(0.5){\modxs}
)
)
,-[ur]
(-[ur]
(-[ur]_>*!LU(0.3){\modms}
)
)
)  
}}
\newcommand*{\Modpenttwoc}{\copy \modpenttwoc}

\newbox \modpenttwocc
\setbox \modpenttwocc =\hbox{\setlabelmargin{2pt}
\xygraph{ !{0;/r0.8pc/:;/u0.8pc/::}[]
     -[u]_<*+!DL(.5){\modns}
(-[ul]^-*!LD(-0.3){\modxs}
(-[ul]^-*!LD(-0.3){\idIs}
(-[ur]_>*!LU(0.5){\modxs}
,-[ul]^>*!RU(0.5){\modxs}
)
,-[ur]
(-[ur]_>*!LU(0.5){\modxs}
)
)
,-[ur]
(-[ur]
(-[ur]_>*!LU(0.3){\modms}
)
)
)  
}}
\newcommand*{\Modpenttwocc}{\copy \modpenttwocc}

\newbox \modpenttwod
\setbox \modpenttwod =\hbox{\setlabelmargin{2pt}
\xygraph{ !{0;/r0.8pc/:;/u0.8pc/::}[]
     -[u]_<*+!DL(.5){\modns}
(-[ul]
(-[ul]
(-[ul]^>*!RU(0.5){\modxs}
)
)
,-[ur]_*!LD(-0.4){\modns}
(-[ur]
(-[ur]_>*!LU(0.3){\modms}
)
,-[ul]^-*!LD(-0.3){\modxs}
(-[ul]^>*!RU(0.5){\modxs}
,-[ur]_>*!LU(0.5){\modxs}
)
)
) 
}}
\newcommand*{\Modpenttwod}{\copy \modpenttwod}

\newbox \modpenttwoe
\setbox \modpenttwoe =\hbox{\setlabelmargin{2pt}
\xygraph{ !{0;/r0.8pc/:;/u0.8pc/::}[]
     -[u]_<*+!DL(.5){\modns}
(-[ul]^-*!LD(-0.3){\modxs}
(-[ul]
(-[ul]^>*!RU(0.5){\modxs}
)
,-[ur]_>*!LD(2.5){\modxs}
(-[ur]_>*!LU(0.5){\modxs}
,-[ul]^>*!RU(0.5){\modxs}
)
)
,-[ur]
(-[ur]
(-[ur]_>*!LU(0.5){\modms}
)
)
)
}}
\newcommand*{\Modpenttwoe}{\copy \modpenttwoe}


\newbox\modpentthreea
\setbox \modpentthreea =\hbox{\xygraph{ !{0;/r0.8pc/:;/u0.8pc/::}[]
   -[u]_<*!DL(.5){\modns}
(-[ul]
(-[ul]
(-[ul]^>*!RU{\modxs}
)
)
)
(-[ur]_>*!LD(2.8){\modms}
(-[ul]
(-[ul]^>*!RU{\modxs}
)
)
(-[ur]_>*!LD(3.5){\modns}
(-[ul]^>*!RU{\modxs}
,-[ur]_>*!LU{\modms}
)
)
)
}}
\newcommand*{\Modpentthreea}{\copy \modpentthreea}

\newbox \modpentthreeb
\setbox \modpentthreeb =\hbox{\setlabelmargin{2pt}
\xygraph{ !{0;/r0.8pc/:;/u0.8pc/::}[]
     -[u]_<*+!DL(.5){\modns}
(-[ul]^>*!LD(-0.3){\modxs}
(-[ul]
(-[ul]^>*!RU(0.5){\modxs}
,-[ur]_>*!LU(0.5){\modxs}
)
)
,-[ur]_>*!RU(-0.1){\modns}
(-[ur]
(-[ur]_>*!LU(0.3){\modms}
,-[ul]^>*!RU(0.5){\modxs}
)
)
)
}}
\newcommand*{\Modpentthreeb}{\copy \modpentthreeb}

\newbox \modpentthreebb
\setbox \modpentthreebb =\hbox{\setlabelmargin{2pt}
\xygraph{ !{0;/r0.8pc/:;/u0.8pc/::}[]
     -[u]_<*+!DL(.5){\modns}
(-[ul]^>*!LD(-0.3){\idIs}
(-[ul]
(-[ul]^>*!RU(0.5){\modxs}
,-[ur]_>*!LU(0.5){\modxs}
)
)
,-[ur]_>*!RU(-0.1){\modns}
(-[ur]
(-[ur]_>*!LU(0.3){\modms}
,-[ul]^>*!RU(0.5){\modxs}
)
)
)
}}
\newcommand*{\Modpentthreebb}{\copy \modpentthreebb}

\newbox \modpentthreec
\setbox \modpentthreec =\hbox{\setlabelmargin{2pt}
\xygraph{ !{0;/r0.8pc/:;/u0.8pc/::}[]
     -[u]_<*+!DL(.5){\modns}
(-[ul]^-*!LD(-0.3){\modxs}
(-[ul]^-*!LD(-0.3){\modxs}
(-[ur]_>*!LU(0.5){\modxs}
,-[ul]^>*!RU(0.5){\modxs}
)
,-[ur]
(-[ur]_>*!LU(0.5){\modxs}
)
)
,-[ur]
(-[ur]
(-[ur]_>*!LU(0.3){\modms}
)
)
)  
}}
\newcommand*{\Modpentthreec}{\copy \modpentthreec}

\newbox \modpentthreecc
\setbox \modpentthreecc =\hbox{\setlabelmargin{2pt}
\xygraph{ !{0;/r0.8pc/:;/u0.8pc/::}[]
     -[u]_<*+!DL(.5){\modns}
(-[ul]^-*!LD(-0.3){\modxs}
(-[ul]^-*!LD(-0.3){\idIs}
(-[ur]_>*!LU(0.5){\modxs}
,-[ul]^>*!RU(0.5){\modxs}
)
,-[ur]
(-[ur]_>*!LU(0.5){\modxs}
)
)
,-[ur]
(-[ur]
(-[ur]_>*!LU(0.3){\modms}
)
)
)  
}}
\newcommand*{\Modpentthreecc}{\copy \modpentthreecc}

\newbox\modpentthreed
\setbox\modpentthreed =\hbox{\setlabelmargin{2pt}
\xygraph{ !{0;/r0.8pc/:;/u0.8pc/::}[]
     -[u]_<*+!DL(.5){\modns}
(-[ul]
(-[ul]
(-[ul]^>*!RU(0.5){\modxs}
)
)
,-[ur]_*!LD(-0.4){\modms}
(-[ur]
(-[ur]_>*!LU(0.3){\modms}
)
,-[ul]^-*!LD(-0.3){\idIs}
(-[ul]^>*!RU(0.5){\modxs}
,-[ur]_>*!LU(0.5){\modxs}
)
)
) 
}}
\newcommand*{\Modpentthreed}{\copy \modpentthreed}

\newbox\modpentthreee
\setbox \modpentthreee =\hbox{\setlabelmargin{2pt}
\xygraph{ !{0;/r0.8pc/:;/u0.8pc/::}[]
     -[u]_<*+!DL(.5){\modns}
(-[ul]^-*!LD(-0.3){\modxs}
(-[ul]
(-[ul]^>*!RU(0.5){\modxs}
)
,-[ur]_<*!RU(3.8){\idIs}
(-[ur]_>*!LU(0.5){\modxs}
,-[ul]^>*!RU(0.5){\modxs}
)
)
,-[ur]
(-[ur]
(-[ur]_>*!LU(0.5){\modms}
)
)
)
}}
\newcommand*{\Modpentthreee}{\copy \modpentthreee}


\newbox\modpentfoura
\setbox \modpentfoura =\hbox{\xygraph{ !{0;/r0.8pc/:;/u0.8pc/::}[]
   -[u]_<*!DL(.5){\modns}
(-[ul]
(-[ul]
(-[ul]^>*!RU{\modxs}
)
)
)
(-[ur]_>*!LD(2.8){\modns}
(-[ul]
(-[ul]^>*!RU{\modxs}
)
)
(-[ur]_>*!LD(3.5){\modns}
(-[ul]^>*!RU{\modxs}
,-[ur]_>*!LU{\modns}
)
)
)
}}
\newcommand*{\Modpentfoura}{\copy \modpentfoura}

\newbox \modpentfourb
\setbox \modpentfourb =\hbox{\setlabelmargin{2pt}
\xygraph{ !{0;/r0.8pc/:;/u0.8pc/::}[]
     -[u]_<*+!DL(.5){\modns}
(-[ul]^>*!LD(-0.3){\modxs}
(-[ul]
(-[ul]^>*!RU(0.5){\modxs}
,-[ur]_>*!LU(0.5){\modxs}
)
)
,-[ur]_>*!RU(-0.1){\modns}
(-[ur]
(-[ur]_>*!LU(0.3){\modns}
,-[ul]^>*!RU(0.5){\modxs}
)
)
)
}}
\newcommand*{\Modpentfourb}{\copy \modpentfourb}

\newbox \modpentfourbb
\setbox \modpentfourbb =\hbox{\setlabelmargin{2pt}
\xygraph{ !{0;/r0.8pc/:;/u0.8pc/::}[]
     -[u]_<*+!DL(.5){\modns}
(-[ul]^>*!LD(-0.3){\idIs}
(-[ul]
(-[ul]^>*!RU(0.5){\modxs}
,-[ur]_>*!LU(0.5){\modxs}
)
)
,-[ur]_>*!RU(-0.1){\modns}
(-[ur]
(-[ur]_>*!LU(0.3){\modns}
,-[ul]^>*!RU(0.5){\modxs}
)
)
)
}}
\newcommand*{\Modpentfourbb}{\copy \modpentfourbb}

\newbox \modpentfourc
\setbox \modpentfourc =\hbox{\setlabelmargin{2pt}
\xygraph{ !{0;/r0.8pc/:;/u0.8pc/::}[]
     -[u]_<*+!DL(.5){\modns}
(-[ul]^-*!LD(-0.3){\modxs}
(-[ul]^-*!LD(-0.3){\modxs}
(-[ur]_>*!LU(0.5){\modxs}
,-[ul]^>*!RU(0.5){\modxs}
)
,-[ur]
(-[ur]_>*!LU(0.5){\modxs}
)
)
,-[ur]
(-[ur]
(-[ur]_>*!LU(0.3){\modns}
)
)
)  
}}
\newcommand*{\Modpentfourc}{\copy \modpentfourc}

\newbox \modpentfourcc
\setbox \modpentfourcc =\hbox{\setlabelmargin{2pt}
\xygraph{ !{0;/r0.8pc/:;/u0.8pc/::}[]
     -[u]_<*+!DL(.5){\modns}
(-[ul]^-*!LD(-0.3){\modxs}
(-[ul]^-*!LD(-0.3){\idIs}
(-[ur]_>*!LU(0.5){\modxs}
,-[ul]^>*!RU(0.5){\modxs}
)
,-[ur]
(-[ur]_>*!LU(0.5){\modxs}
)
)
,-[ur]
(-[ur]
(-[ur]_>*!LU(0.3){\modns}
)
)
)  
}}
\newcommand*{\Modpentfourcc}{\copy \modpentfourcc}

\newbox \modpentfourd
\setbox \modpentfourd =\hbox{\setlabelmargin{2pt}
\xygraph{ !{0;/r0.8pc/:;/u0.8pc/::}[]
     -[u]_<*+!DL(.5){\modns}
(-[ul]
(-[ul]
(-[ul]^>*!RU(0.5){\modxs}
)
)
,-[ur]_*!LD(-0.4){\modns}
(-[ur]
(-[ur]_>*!LU(0.3){\modns}
)
,-[ul]^-*!LD(-0.3){\modxs}
(-[ul]^>*!RU(0.5){\modxs}
,-[ur]_>*!LU(0.5){\modxs}
)
)
) 
}}
\newcommand*{\Modpentfourd}{\copy \modpentfourd}

\newbox \modpentfourdd
\setbox \modpentfourdd =\hbox{\setlabelmargin{2pt}
\xygraph{ !{0;/r0.8pc/:;/u0.8pc/::}[]
     -[u]_<*+!DL(.5){\modns}
(-[ul]
(-[ul]
(-[ul]^>*!RU(0.5){\modxs}
)
)
,-[ur]_*!LD(-0.4){\modns}
(-[ur]
(-[ur]_>*!LU(0.3){\modns}
)
,-[ul]^-*!LD(-0.3){\idIs}
(-[ul]^>*!RU(0.5){\modxs}
,-[ur]_>*!LU(0.5){\modxs}
)
)
) 
}}
\newcommand*{\Modpentfourdd}{\copy \modpentfourdd}

\newbox \modpentfoure
\setbox \modpentfoure =\hbox{\setlabelmargin{2pt}
\xygraph{ !{0;/r0.8pc/:;/u0.8pc/::}[]
     -[u]_<*+!DL(.5){\modns}
(-[ul]^-*!LD(-0.3){\modxs}
(-[ul]
(-[ul]^>*!RU(0.5){\modxs}
)
,-[ur]_>*!LD(2.5){\modxs}
(-[ur]_>*!LU(0.5){\modxs}
,-[ul]^>*!RU(0.5){\modxs}
)
)
,-[ur]
(-[ur]
(-[ur]_>*!LU(0.5){\modns}
)
)
)
}}
\newcommand*{\Modpentfoure}{\copy \modpentfoure}

\newbox \modpentfouree
\setbox \modpentfouree =\hbox{\setlabelmargin{2pt}
\xygraph{ !{0;/r0.8pc/:;/u0.8pc/::}[]
     -[u]_<*+!DL(.5){\modns}
(-[ul]^-*!LD(-0.3){\modxs}
(-[ul]
(-[ul]^>*!RU(0.5){\modxs}
)
,-[ur]_<*!RU(3.8){\idIs}
(-[ur]_>*!LU(0.5){\modxs}
,-[ul]^>*!RU(0.5){\modxs}
)
)
,-[ur]
(-[ur]
(-[ur]_>*!LU(0.5){\modns}
)
)
)
}}
\newcommand*{\Modpentfouree}{\copy \modpentfouree}


\newbox\modpentfivea
\setbox \modpentfivea =\hbox{\xygraph{ !{0;/r0.8pc/:;/u0.8pc/::}[]
   -[u]_<*!DL(.5){\modns}
(-[ul]
(-[ul]
(-[ul]^>*!RU{\modxs}
)
)
)
(-[ur]_>*!LD(2.8){\modns}
(-[ul]
(-[ul]^>*!RU{\modxs}
)
)
(-[ur]_>*!LD(2.5){\modms}
(-[ul]^>*!RU{\modxs}
,-[ur]_>*!LU{\modns}
)
)
)
}}
\newcommand*{\Modpentfivea}{\copy \modpentfivea}

\newbox \modpentfiveb
\setbox \modpentfiveb =\hbox{\setlabelmargin{2pt}
\xygraph{ !{0;/r0.8pc/:;/u0.8pc/::}[]
     -[u]_<*+!DL(.5){\modns}
(-[ul]^>*!LD(-0.3){\modxs}
(-[ul]
(-[ul]^>*!RU(0.5){\modxs}
,-[ur]_>*!LU(0.5){\modxs}
)
)
,-[ur]_>*!RU(-0.1){\modms}
(-[ur]
(-[ur]_>*!LU(0.3){\modns}
,-[ul]^>*!RU(0.5){\modxs}
)
)
)
}}
\newcommand*{\Modpentfiveb}{\copy \modpentfiveb}

\newbox \modpentfivebb
\setbox \modpentfivebb =\hbox{\setlabelmargin{2pt}
\xygraph{ !{0;/r0.8pc/:;/u0.8pc/::}[]
     -[u]_<*+!DL(.5){\modns}
(-[ul]^>*!LD(-0.3){\idIs}
(-[ul]
(-[ul]^>*!RU(0.5){\modxs}
,-[ur]_>*!LU(0.5){\modxs}
)
)
,-[ur]_>*!RU(-0.1){\modns}
(-[ur]
(-[ur]_>*!LU(0.3){\modms}
,-[ul]^>*!RU(0.5){\modxs}
)
)
)
}}
\newcommand*{\Modpentfivebb}{\copy \modpentfivebb}

\newbox \modpentfivec
\setbox \modpentfivec =\hbox{\setlabelmargin{2pt}
\xygraph{ !{0;/r0.8pc/:;/u0.8pc/::}[]
     -[u]_<*+!DL(.5){\modns}
(-[ul]^-*!LD(-0.3){\modxs}
(-[ul]^-*!LD(-0.3){\modxs}
(-[ur]_>*!LU(0.5){\modxs}
,-[ul]^>*!RU(0.5){\modxs}
)
,-[ur]
(-[ur]_>*!LU(0.5){\modxs}
)
)
,-[ur]
(-[ur]
(-[ur]_>*!LU(0.3){\modms}
)
)
)  
}}
\newcommand*{\Modpentfivec}{\copy \modpentfivec}

\newbox \modpentfivecc
\setbox \modpentfivecc =\hbox{\setlabelmargin{2pt}
\xygraph{ !{0;/r0.8pc/:;/u0.8pc/::}[]
     -[u]_<*+!DL(.5){\modns}
(-[ul]^-*!LD(-0.3){\modxs}
(-[ul]^-*!LD(-0.3){\idIs}
(-[ur]_>*!LU(0.5){\modxs}
,-[ul]^>*!RU(0.5){\modxs}
)
,-[ur]
(-[ur]_>*!LU(0.5){\modxs}
)
)
,-[ur]
(-[ur]
(-[ur]_>*!LU(0.3){\modms}
)
)
)  
}}
\newcommand*{\Modpentfivecc}{\copy \modpentfivecc}

\newbox \modpentfived
\setbox \modpentfived =\hbox{\setlabelmargin{2pt}
\xygraph{ !{0;/r0.8pc/:;/u0.8pc/::}[]
     -[u]_<*+!DL(.5){\modns}
(-[ul]
(-[ul]
(-[ul]^>*!RU(0.5){\modxs}
)
)
,-[ur]_*!LD(-0.4){\modms}
(-[ur]
(-[ur]_>*!LU(0.3){\modms}
)
,-[ul]^-*!LD(-0.3){\idIs}
(-[ul]^>*!RU(0.5){\modxs}
,-[ur]_>*!LU(0.5){\modxs}
)
)
) 
}}
\newcommand*{\Modpentfived}{\copy \modpentfived}

\newbox \modpentfivee
\setbox \modpentfivee =\hbox{\setlabelmargin{2pt}
\xygraph{ !{0;/r0.8pc/:;/u0.8pc/::}[]
     -[u]_<*+!DL(.5){\modns}
(-[ul]^-*!LD(-0.3){\modxs}
(-[ul]
(-[ul]^>*!RU(0.5){\modxs}
)
,-[ur]_<*!RU(3.8){\idIs}
(-[ur]_>*!LU(0.5){\modxs}
,-[ul]^>*!RU(0.5){\modxs}
)
)
,-[ur]
(-[ur]
(-[ur]_>*!LU(0.5){\modms}
)
)
)
}}
\newcommand*{\Modpentfivee}{\copy \modpentfivee}


\newbox\modpentsixa
\setbox \modpentsixa =\hbox{\xygraph{ !{0;/r0.8pc/:;/u0.8pc/::}[]
   -[u]_<*!DL(.5){\modns}
(-[ul]
(-[ul]
(-[ul]^>*!RU{\modxs}
)
)
)
(-[ur]_>*!LD(2.8){\modms}
(-[ul]
(-[ul]^>*!RU{\modxs}
)
)
(-[ur]_>*!LD(3.5){\modns}
(-[ul]^>*!RU{\modxs}
,-[ur]_>*!LU{\modns}
)
)
)
}}
\newcommand*{\Modpentsixa}{\copy \modpentsixa}

\newbox \modpentsixb
\setbox \modpentsixb =\hbox{\setlabelmargin{2pt}
\xygraph{ !{0;/r0.8pc/:;/u0.8pc/::}[]
     -[u]_<*+!DL(.5){\modns}
(-[ul]^>*!LD(-0.3){\modxs}
(-[ul]
(-[ul]^>*!RU(0.5){\modxs}
,-[ur]_>*!LU(0.5){\modxs}
)
)
,-[ur]_>*!RU(-0.1){\modns}
(-[ur]
(-[ur]_>*!LU(0.3){\modms}
,-[ul]^>*!RU(0.5){\modxs}
)
)
)
}}
\newcommand*{\Modpentsixb}{\copy \modpentsixb}

\newbox \modpentsixbb
\setbox \modpentsixbb =\hbox{\setlabelmargin{2pt}
\xygraph{ !{0;/r0.8pc/:;/u0.8pc/::}[]
     -[u]_<*+!DL(.5){\modns}
(-[ul]^>*!LD(-0.3){\idIs}
(-[ul]
(-[ul]^>*!RU(0.5){\modxs}
,-[ur]_>*!LU(0.5){\modxs}
)
)
,-[ur]_>*!RU(-0.1){\modns}
(-[ur]
(-[ur]_>*!LU(0.3){\modms}
,-[ul]^>*!RU(0.5){\modxs}
)
)
)
}}
\newcommand*{\Modpentsixbb}{\copy \modpentsixbb}

\newbox \modpentsixc
\setbox \modpentsixc =\hbox{\setlabelmargin{2pt}
\xygraph{ !{0;/r0.8pc/:;/u0.8pc/::}[]
     -[u]_<*+!DL(.5){\modns}
(-[ul]^-*!LD(-0.3){\idIs}
(-[ul]^-*!LD(-0.3){\modxs}
(-[ur]_>*!LU(0.5){\modxs}
,-[ul]^>*!RU(0.5){\modxs}
)
,-[ur]
(-[ur]_>*!LU(0.5){\modxs}
)
)
,-[ur]
(-[ur]
(-[ur]_>*!LU(0.3){\modns}
)
)
)  
}}
\newcommand*{\Modpentsixc}{\copy \modpentsixc}

\newbox \modpentsixcc
\setbox \modpentsixcc =\hbox{\setlabelmargin{2pt}
\xygraph{ !{0;/r0.8pc/:;/u0.8pc/::}[]
     -[u]_<*+!DL(.5){\modns}
(-[ul]^-*!LD(-0.3){\idIs}
(-[ul]^-*!LD(-0.3){\modxs}
(-[ur]_>*!LU(0.5){\modxs}
,-[ul]^>*!RU(0.5){\modxs}
)
,-[ur]
(-[ur]_>*!LU(0.5){\modxs}
)
)
,-[ur]
(-[ur]
(-[ur]_>*!LU(0.3){\modns}
)
)
)  
}}
\newcommand*{\Modpentsixcc}{\copy \modpentsixcc}

\newbox \modpentsixd
\setbox \modpentsixd =\hbox{\setlabelmargin{2pt}
\xygraph{ !{0;/r0.8pc/:;/u0.8pc/::}[]
     -[u]_<*+!DL(.5){\modns}
(-[ul]
(-[ul]
(-[ul]^>*!RU(0.5){\modxs}
)
)
,-[ur]_*!LD(-0.4){\modms}
(-[ur]
(-[ur]_>*!LU(0.3){\modns}
)
,-[ul]^-*!LD(-0.3){\modxs}
(-[ul]^>*!RU(0.5){\modxs}
,-[ur]_>*!LU(0.5){\modxs}
)
)
) 
}}
\newcommand*{\Modpentsixd}{\copy \modpentsixd}

\newbox \modpentsixe
\setbox \modpentsixe =\hbox{\setlabelmargin{2pt}
\xygraph{ !{0;/r0.8pc/:;/u0.8pc/::}[]
     -[u]_<*+!DL(.5){\modns}
(-[ul]^-*!LD(-0.3){\idIs}
(-[ul]
(-[ul]^>*!RU(0.5){\modxs}
)
,-[ur]_>*!LD(2.5){\modxs}
(-[ur]_>*!LU(0.5){\modxs}
,-[ul]^>*!RU(0.5){\modxs}
)
)
,-[ur]
(-[ur]
(-[ur]_>*!LU(0.5){\modns}
)
)
)
}}
\newcommand*{\Modpentsixe}{\copy \modpentsixe}


\newbox\modpentsevena
\setbox \modpentsevena =\hbox{\xygraph{ !{0;/r0.8pc/:;/u0.8pc/::}[]
   -[u]_<*!DL(.5){\modms}
(-[ul]
(-[ul]
(-[ul]^>*!RU{\modxs}
)
)
)
(-[ur]_>*!LD(2.8){\modns}
(-[ul]
(-[ul]^>*!RU{\modxs}
)
)
(-[ur]_>*!LD(2.8){\modms}
(-[ul]^>*!RU{\modxs}
,-[ur]_>*!LU{\modns}
)
)
)
}}
\newcommand*{\Modpentsevena}{\copy \modpentsevena}

\newbox \modpentsevenb
\setbox \modpentsevenb =\hbox{\setlabelmargin{2pt}
\xygraph{ !{0;/r0.8pc/:;/u0.8pc/::}[]
     -[u]_<*+!DL(.5){\modms}
(-[ul]^>*!LD(-0.3){\modxs}
(-[ul]
(-[ul]^>*!RU(0.5){\modxs}
,-[ur]_>*!LU(0.5){\modxs}
)
)
,-[ur]_>*!RU(-0.1){\modns}
(-[ur]
(-[ur]_>*!LU(0.3){\modns}
,-[ul]^>*!RU(0.5){\modxs}
)
)
)
}}
\newcommand*{\Modpentsevenb}{\copy \modpentsevenb}

\newbox \modpentsevenbb
\setbox \modpentsevenbb =\hbox{\setlabelmargin{2pt}
\xygraph{ !{0;/r0.8pc/:;/u0.8pc/::}[]
     -[u]_<*+!DL(.5){\modms}
(-[ul]^>*!LD(-0.3){\idIs}
(-[ul]
(-[ul]^>*!RU(0.5){\modxs}
,-[ur]_>*!LU(0.5){\modxs}
)
)
,-[ur]_>*!RU(-0.1){\modms}
(-[ur]
(-[ur]_>*!LU(0.3){\modns}
,-[ul]^>*!RU(0.5){\modxs}
)
)
)
}}
\newcommand*{\Modpentsevenbb}{\copy \modpentsevenbb}

\newbox \modpentsevenc
\setbox \modpentsevenc =\hbox{\setlabelmargin{2pt}
\xygraph{ !{0;/r0.8pc/:;/u0.8pc/::}[]
     -[u]_<*+!DL(.5){\modms}
(-[ul]^-*!LD(-0.3){\modxs}
(-[ul]^-*!LD(-0.3){\modxs}
(-[ur]_>*!LU(0.5){\modxs}
,-[ul]^>*!RU(0.5){\modxs}
)
,-[ur]
(-[ur]_>*!LU(0.5){\modxs}
)
)
,-[ur]
(-[ur]
(-[ur]_>*!LU(0.3){\modns}
)
)
)  
}}
\newcommand*{\Modpentsevenc}{\copy \modpentsevenc}

\newbox \modpentsevencc
\setbox \modpentsevencc =\hbox{\setlabelmargin{2pt}
\xygraph{ !{0;/r0.8pc/:;/u0.8pc/::}[]
     -[u]_<*+!DL(.5){\modms}
(-[ul]^-*!LD(-0.3){\modxs}
(-[ul]^-*!LD(-0.3){\idIs}
(-[ur]_>*!LU(0.5){\modxs}
,-[ul]^>*!RU(0.5){\modxs}
)
,-[ur]
(-[ur]_>*!LU(0.5){\modxs}
)
)
,-[ur]
(-[ur]
(-[ur]_>*!LU(0.3){\modns}
)
)
)  
}}
\newcommand*{\Modpentsevencc}{\copy \modpentsevencc}

\newbox \modpentsevend
\setbox \modpentsevend =\hbox{\setlabelmargin{2pt}
\xygraph{ !{0;/r0.8pc/:;/u0.8pc/::}[]
     -[u]_<*+!DL(.5){\modms}
(-[ul]
(-[ul]
(-[ul]^>*!RU(0.5){\modxs}
)
)
,-[ur]_*!LD(-0.4){\modns}
(-[ur]
(-[ur]_>*!LU(0.3){\modns}
)
,-[ul]^-*!LD(-0.3){\modxs}
(-[ul]^>*!RU(0.5){\modxs}
,-[ur]_>*!LU(0.5){\modxs}
)
)
) 
}}
\newcommand*{\Modpentsevend}{\copy \modpentsevend}

\newbox \modpentsevendd
\setbox \modpentsevendd =\hbox{\setlabelmargin{2pt}
\xygraph{ !{0;/r0.8pc/:;/u0.8pc/::}[]
     -[u]_<*+!DL(.5){\modms}
(-[ul]
(-[ul]
(-[ul]^>*!RU(0.5){\modxs}
)
)
,-[ur]_*!LD(-0.4){\modns}
(-[ur]
(-[ur]_>*!LU(0.3){\modns}
)
,-[ul]^-*!LD(-0.3){\idIs}
(-[ul]^>*!RU(0.5){\modxs}
,-[ur]_>*!LU(0.5){\modxs}
)
)
) 
}}
\newcommand*{\Modpentsevendd}{\copy \modpentsevendd}

\newbox \modpentsevene
\setbox \modpentsevene =\hbox{\setlabelmargin{2pt}
\xygraph{ !{0;/r0.8pc/:;/u0.8pc/::}[]
     -[u]_<*+!DL(.5){\modms}
(-[ul]^-*!LD(-0.3){\modxs}
(-[ul]
(-[ul]^>*!RU(0.5){\modxs}
)
,-[ur]_>*!LD(2.5){\modxs}
(-[ur]_>*!LU(0.5){\modxs}
,-[ul]^>*!RU(0.5){\modxs}
)
)
,-[ur]
(-[ur]
(-[ur]_>*!LU(0.5){\modns}
)
)
)
}}
\newcommand*{\Modpentsevene}{\copy \modpentsevene}

\newbox \modpentsevenee
\setbox \modpentsevenee =\hbox{\setlabelmargin{2pt}
\xygraph{ !{0;/r0.8pc/:;/u0.8pc/::}[]
     -[u]_<*+!DL(.5){\modms}
(-[ul]^-*!LD(-0.3){\modxs}
(-[ul]
(-[ul]^>*!RU(0.5){\modxs}
)
,-[ur]_<*!RU(3.8){\idIs}
(-[ur]_>*!LU(0.5){\modxs}
,-[ul]^>*!RU(0.5){\modxs}
)
)
,-[ur]
(-[ur]
(-[ur]_>*!LU(0.5){\modns}
)
)
)
}}
\newcommand*{\Modpentsevenee}{\copy \modpentsevenee}


\newbox\modpenteighta
\setbox \modpenteighta =\hbox{\xygraph{ !{0;/r0.8pc/:;/u0.8pc/::}[]
   -[u]_<*!DL(.5){\modms}
(-[ul]
(-[ul]
(-[ul]^>*!RU{\modxs}
)
)
)
(-[ur]_>*!LD(2.8){\modns}
(-[ul]
(-[ul]^>*!RU{\modxs}
)
)
(-[ur]_>*!LD(3.5){\modns}
(-[ul]^>*!RU{\modxs}
,-[ur]_>*!LU{\modns}
)
)
)
}}
\newcommand*{\Modpenteighta}{\copy \modpenteighta}

\newbox \modpenteightb
\setbox \modpenteightb =\hbox{\setlabelmargin{2pt}
\xygraph{ !{0;/r0.8pc/:;/u0.8pc/::}[]
     -[u]_<*+!DL(.5){\modns}
(-[ul]^>*!LD(-0.3){\modxs}
(-[ul]
(-[ul]^>*!RU(0.5){\modxs}
,-[ur]_>*!LU(0.5){\modxs}
)
)
,-[ur]_>*!RU(-0.1){\modns}
(-[ur]
(-[ur]_>*!LU(0.3){\modms}
,-[ul]^>*!RU(0.5){\modxs}
)
)
)
}}
\newcommand*{\Modpenteightb}{\copy \modpenteightb}

\newbox \modpenteightbb
\setbox \modpenteightbb =\hbox{\setlabelmargin{2pt}
\xygraph{ !{0;/r0.8pc/:;/u0.8pc/::}[]
     -[u]_<*+!DL(.5){\modms}
(-[ul]^>*!LD(-0.3){\modxs}
(-[ul]
(-[ul]^>*!RU(0.5){\modxs}
,-[ur]_>*!LU(0.5){\modxs}
)
)
,-[ur]_>*!RU(-0.1){\modns}
(-[ur]
(-[ur]_>*!LU(0.3){\modns}
,-[ul]^>*!RU(0.5){\modxs}
)
)
)
}}
\newcommand*{\Modpenteightbb}{\copy \modpenteightbb}

\newbox \modpenteightc
\setbox \modpenteightc =\hbox{\setlabelmargin{2pt}
\xygraph{ !{0;/r0.8pc/:;/u0.8pc/::}[]
     -[u]_<*+!DL(.5){\modns}
(-[ul]^-*!LD(-0.3){\modxs}
(-[ul]^-*!LD(-0.3){\modxs}
(-[ur]_>*!LU(0.5){\modxs}
,-[ul]^>*!RU(0.5){\modxs}
)
,-[ur]
(-[ur]_>*!LU(0.5){\modxs}
)
)
,-[ur]
(-[ur]
(-[ur]_>*!LU(0.3){\modms}
)
)
)  
}}
\newcommand*{\Modpenteightc}{\copy \modpenteightc}

\newbox \modpenteightcc
\setbox \modpenteightcc =\hbox{\setlabelmargin{2pt}
\xygraph{ !{0;/r0.8pc/:;/u0.8pc/::}[]
     -[u]_<*+!DL(.5){\modms}
(-[ul]^-*!LD(-0.3){\modxs}
(-[ul]^-*!LD(-0.3){\modxs}
(-[ur]_>*!LU(0.5){\modxs}
,-[ul]^>*!RU(0.5){\modxs}
)
,-[ur]
(-[ur]_>*!LU(0.5){\modxs}
)
)
,-[ur]
(-[ur]
(-[ur]_>*!LU(0.3){\modns}
)
)
)  
}}
\newcommand*{\Modpenteightcc}{\copy \modpenteightcc}

\newbox \modpenteightd
\setbox \modpenteightd =\hbox{\setlabelmargin{2pt}
\xygraph{ !{0;/r0.8pc/:;/u0.8pc/::}[]
     -[u]_<*+!DL(.5){\modms}
(-[ul]
(-[ul]
(-[ul]^>*!RU(0.5){\modxs}
)
)
,-[ur]_*!LD(-0.4){\modns}
(-[ur]
(-[ur]_>*!LU(0.3){\modns}
)
,-[ul]^-*!LD(-0.3){\modxs}
(-[ul]^>*!RU(0.5){\modxs}
,-[ur]_>*!LU(0.5){\modxs}
)
)
) 
}}
\newcommand*{\Modpenteightd}{\copy \modpenteightd}

\newbox \modpenteightdd
\setbox \modpenteightdd =\hbox{\setlabelmargin{2pt}
\xygraph{ !{0;/r0.8pc/:;/u0.8pc/::}[]
     -[u]_<*+!DL(.5){\modms}
(-[ul]
(-[ul]
(-[ul]^>*!RU(0.5){\modxs}
)
)
,-[ur]_*!LD(-0.4){\modns}
(-[ur]
(-[ur]_>*!LU(0.3){\modns}
)
,-[ul]^-*!LD(-0.3){\idIs}
(-[ul]^>*!RU(0.5){\modxs}
,-[ur]_>*!LU(0.5){\modxs}
)
)
) 
}}
\newcommand*{\Modpenteightdd}{\copy \modpenteightdd}

\newbox \modpenteighte
\setbox \modpenteighte =\hbox{\setlabelmargin{2pt}
\xygraph{ !{0;/r0.8pc/:;/u0.8pc/::}[]
     -[u]_<*+!DL(.5){\modns}
(-[ul]^-*!LD(-0.3){\modxs}
(-[ul]
(-[ul]^>*!RU(0.5){\modxs}
)
,-[ur]_<*!RU(3.8){\idIs}
(-[ur]_>*!LU(0.5){\modxs}
,-[ul]^>*!RU(0.5){\modxs}
)
)
,-[ur]
(-[ur]
(-[ur]_>*!LU(0.5){\modms}
)
)
)
}}
\newcommand*{\Modpenteighte}{\copy \modpenteighte}

\author{Alexei Davydov\footnote{Department of Mathematics and Statistics, University of New Hampshire, Durham, NH 03824, USA; alexei1davydov@gmail.com}\ \ and Tom Booker\footnote{Department of Mathematics, Faculty of Science, Macquarie University, Sydney, NSW 2109, Australia; tom.booker@mq.edu.au}\, \footnote{The author was supported by an Australian Postgraduate Award.}}
\title{Commutative Algebras in Fibonacci Categories}
\maketitle
\date{}

\begin{abstract}
By studying NIM-representations we show that the Fibonacci category and its tensor powers are completely anisotropic; that is, they do not have any non-trivial separable commutative ribbon algebras.

As an application we deduce that a chiral algebra with the representation category equivalent to a product of Fibonacci categories is maximal; that is, it is not a proper subalgebra of another chiral algebra.
In particular the chiral algebras of the Yang-Lee model, the WZW models of $G_2$ and $F_4$ at level 1, as well as their tensor powers, are maximal.
\end{abstract}
\tableofcontents

\section{Introduction}

Application of category theory to rational conformal field theory has a long history. Representations of chiral symmetries of a rational conformal field theory form a certain kind of braided tensor category known as a {\em modular} category \cite{ms,Hu} and many features of conformal field theory such as boundary extensions, bulk field space, symmetries of conformal field theory have neat categorical interpretations. In particular category theory is well suited for studying extensions of chiral algebras. According to A. Kirillov, Y.-Z. Huang and J. Lepowsky (see \cite{KiO} and references therein) a chiral extension of a chiral algebra $\V$ corresponds to a certain commutative algebra in the modular category of representations of $\V$. More precisely the commutative algebras in question should be simple, separable and ribbon (concepts explained in section \ref{prel}). Moreover the category of representations of the extended chiral algebra can be read off from the corresponding commutative algebra as the category of its local modules. Relatively simple categorical arguments show that there are only a finite number of simple, separable, ribbon, commutative algebras in a given modular category and that all maximal algebras have equivalent categories of local modules (see \cite{dmno} and references therein). This immediately implies that a rational chiral algebra has only a finite number of extensions. Moreover maximal extensions of a rational chiral algebra all have the same representation type. 
All this indicates the special role played by maximal (rational) chiral algebras and their categories of representations. The property that characterises categories  of representations of maximal chiral algebras is the absence of non-trivial separable, ribbon, commutative algebras in the category. Such categories were called {\em completely anisotropic} in \cite{dmno}. 
This simplest class of maximal chiral algebras is formed by holomorphic chiral algebras; that is, chiral algebras with no non-trivial representations. Clearly tensor products of holomorphic algebras are again holomorphic.
Here we look at another class of maximal chiral algebras closed under tensor products. This class (as with the class of holomorphic chiral algebras) is defined by their categories of representations.

In this paper we deal with a specific type of modular category with just two simple objects: $I$ and $X$; and the tensor product decompositions:
$$I\ot I = I,\quad X\ot I = I\ot X = X,\quad X\ot X = I\oplus X.$$
Such categories are called {\em Fibonacci} categories. There are four non-equivalent Fibonacci modular categories $\fib_u$, labelled by primitive roots of unity $u$ of order 10 (section \ref{secfib}). 
We prove that tensor powers $\fib_u^{\boxtimes l}$ of a Fibonacci modular category are completely anisotropic; that is, they do not have non-trivial separable, ribbon, commutative algebras. We do it by describing nonnegative integer matrix (or NIM-) representations of the fusion rules of $\fib_u^{\boxtimes l}$. They give us an insight into possible module categories over $\fib_u^{\boxtimes l}$. Since an algebra in a monoidal category gives rise to a module category, our knowledge of NIM-representations for $\fib_u^{\boxtimes l}$ allows us to prove complete anisotropy of $\fib_u^{\boxtimes l}$. Actually our methods give a stronger result. A product of Fibonacci categories $\fib_u$ is completely anisotropic as long as none of the indices $u$ is inverse to each other. The condition that $uv\not=1$ is not accidental. General category theory implies that the category $\fib_u\boxtimes\fib_{u^{-1}}$ always has a non-trivial simple, separable, ribbon, commutative algebra. 

Complete anisotropy of Fibonacci categories implies that chiral algebras with categories of representations of the form $\fib_u^{\boxtimes l}$ are maximal.
Among such chiral algebras are the chiral algebra $M(2,5)$ from the minimal series; that is, the chiral algebra of the Yang-Lee model (the category of representations is $\fib_u$, with $u=e^{\frac{\pi i}{5}}$), a maximal extension of $M(3,5)^{\times 8}\times M(2,5)^{\times 7}$ (the category of representations is $\fib_u$, with $u=e^{\frac{9\pi i}{5}}$), the affine chiral algebra $G_{2,1}$ (the category of representations is $\fib_u$, with $u=e^{\frac{3\pi i}{5}}$) and the affine chiral algebra $F_{4,1}$ (with the category of representations $\fib_u$ for $u=e^{\frac{7\pi i}{5}}$). Clearly tensoring with holomorphic algebras keeps the class of Fibonacci chiral algebras closed. Among examples of chiral algebras, which are not products of the above mentioned Fibonacci chiral algebras with holomorphic algebras is the coset $F_{4,1}/G_{2,1}$ (the category of representations is $\fib_u\boxtimes\fib_u$ with $u=e^{\frac{7\pi i}{5}}$). This coset can not be written as a tensor product of two chiral algebras of Fibonacci type. 

\section{Preliminaries}\label{prel}

Throughout the paper we assume that the ground field $k$ is an algebraically closed field of characteristic zero (such as the field $\C$ of complex numbers). 

\subsection{Modular categories}

Recall that a rigid monoidal category is called {\em fusion} when it is semi-simple $k$-linear together with a $k$-linear tensor product, finite-dimensional hom spaces and a finite number of simple objects (up to isomorphism). 
We denote the set (of representatives) of isomorphism classes of simple objects in $\cC$ by $Irr(\cC)$. 

Slightly changing the definition from \cite{tu}, we call a fusion category {\em modular} if it is rigid, braided, ribbon and satisfies the non-degeneracy ({\em modularity}) condition: for isomorphism classes of simple objects, the traces of the double braidings form a non-degenerate matrix
$$\tilde{S} = (\tilde{S}_{X,Y})_{X,Y\in Irr(\cC)},\quad \tilde{S}_{X,Y} = tr(c_{X,Y}c_{Y,X}).$$
Here $c_{X,Y}:X\otimes Y\to Y\otimes X$ is the braiding (see \cite{tu,BaKi} for details).

Let $\cC$ be a ribbon category. 
Following \cite{tu} define $\overline{\cC}$ to be $\cC$ as a monoidal category equipped with a new braiding and ribbon twist:
$$\overline{c}_{X,Y}=c_{Y,X}^{-1},\quad \overline{\theta}_X = \theta_X^{-1}.$$
Again it is very easy to see that for a modular $\cC$, $\overline\cC$ is also modular.

Recall that the Deligne tensor product $\cC\boxtimes\caD$ of two fusion categories is a fusion category with simple objects $Irr(\cC\boxtimes\caD) = Irr(\cC)\times Irr(\caD)$ and the tensor product defined by
$$(X\boxtimes Y)\otimes (Z\boxtimes W) = (X\otimes Z)\boxtimes(Y\otimes W).$$
It is straightforward to see that the Deligne tensor product of two modular categories is modular.

Let $\cC$ be a full modular subcategory of a modular category $\caD$. It was proved in \cite{mu1} that as modular categories
\begin{equation}\label{mudec}\caD = \cC\boxtimes\cC_\caD(\cC),\end{equation}
where the category $\cC_\caD(\cC)$ (the {\em M\"uger's centraliser} of $\cC$ in $\caD$) is defined as the full subcategory of $\caD$ of objects {\em transparent} with respect to objects of $\cC$:
$$\cC_\caD(\cC) = \{X\in\caD|\ c_{Y,X}c_{X,Y} = 1_{X\ot Y},\ \forall Y\in\cC\}.$$

Let $\cZ (\cC)$ be the monoidal centre of $\cC$ \cite{js2}. Recall that a braiding $c$ in $\cC$ gives rise to a braided monoidal functor $\iota_+:\cC\to \cZ (\cC)$. 
Using the (conjugate) braiding $\overline c$ we can define another braided monoidal functor $\iota_-:\overline\cC\to \cZ(\cC)$. 
Taking the Deligne tensor product we can combine these two functors into one braided monoidal functor $$\iota:\cC\boxtimes\overline\cC\to \cZ(\cC).$$
The following characterisation of modularity was proven in \cite{mu}. 
A braided fusion category $\cC$ is modular if and only if the functor $\iota$ is an equivalence. 

\begin{example}
We call a fusion category $\cC$ {\em pointed} if all its simple objects are invertible, i.e. $X\otimes X^*\cong 1$ for any simple $X\in\cC$. In this case the set $Irr(\cC)$ of isomorphism classes of simple objects is a group (with respect to the tensor product). Clearly for a braided $\cC$ this group must be abelian.
It was shown in \cite{js} that (up to braided equivalence) braided structures on a pointed category $\cC$ with the group $A = Irr(\cC)$ are in one-to-one correspondence with functions $q:A\to k^*$ ($k^*$ denotes the multiplicative group of $k$) satisfying
$q(a^{-1}) = q(a)$ and such that $\sigma(a,b)= q(ab)q(a)^{-1}q(b)^{-1}$ is bilinear in $a$ and $b$ (a {\em bicharacter}). The correspondence assigns to a braiding $c$ the function $q$ such that $c_{a,a} = q(a)1_{a\otimes a}$. We will denote by $\cC(A,q)$ the pointed category with the group of objects $A$ and the braiding corresponding to $q$. The following are straightforward
$$\cC(A\times B,q_A\times q_B) \simeq \cC(A,q_A)\boxtimes\cC(B,q_B),\quad \overline{\cC(A,q)} \simeq \cC(A,q^{-1}).$$
The braiding of $\cC(A,q)$ is non-degenerate if and only if the form $\sigma$ is non-degenerate:
$$ker(\sigma) = \{a\in A|\ \sigma(a,b)=1\ \forall b\in A\} = 1.$$
Modular structures (ribbon twists) on the braided category $\cC(A,q)$ correspond to homomorphisms $d:A\to \Z/2\Z$ (the dimension function). The ribbon twist $\theta$ corresponding to $d$ has the form
$$\theta_a = d(a)q(a)1_a.$$
\end{example}

\subsection{Module categories, algebras in monoidal categories and their modules}\label{algandmods}

Let $\ev$ be a monoidal category.
A \emph{module category} (see \cite{qu,mkgj,mc}) is a category $\bm$ together with a functor $\ast: \ev\times\bm\to\bm$ and natural families of isomorphisms
$$\alpha_{X,Y,M}: (X\otimes Y)\ast M \to X\ast(Y\ast M)$$
and,
$$\lambda_{M}: I\ast M \to M$$
such that diagrams (1.1), (1.2) and (1.3) of \cite{mkgj} commute.

An (associative, unital) {\em algebra} in a monoidal category $\cC$ is a triple $(A,\mu,\iota)$ consisting of an object $A\in\cC$ together with a {\em multiplication} $\mu:A\otimes A\to A$ and a {\em unit} map $\iota:I\to A$, satisfying {\em associativiy}
$$\mu(\mu\otimes 1) = \mu(1\otimes\mu),$$ 
and {\em unit} 
$$\mu(\iota\otimes 1) = 1 =\mu(1\otimes\iota)$$
axioms. 
Where it will not cause confusion we will be talking about an algebra $A$, suppressing its multiplication and unit maps. A {\em homomorphism} of algebra is a morphism of underlying objects preserving algebra structures in the obvious way. An algebra is {\em simple} if any (non-zero) homomorphism out of it is injective.

A right {\em module} over an algebra $A$ is a pair $(M,\nu)$, where $M$ is an object of $\cC$ and $\nu:M\otimes A\to M$ is a morphism ({\em action map}), such that 
$$\nu(\nu\otimes 1) = \nu(1\otimes\mu).$$ 
A {\em homomorphism} of right $A$-modules $M\to N$ is a morphism $f:M\to N$ in $\cC$ such that 
$$\nu_N(f\otimes 1) = f\nu_M.$$

Right modules over an algebra $A\in\cC$ together with module homomorphisms form a category $\cC_A$.
The forgetful functor $\cC_A\to\cC$ has a right adjoint, which sends an object $X\in\cC$ into the {\em free} $A$-module $X\otimes A$, with $A$-module structure defined by
$$\xymatrix{X\otimes A\otimes A \ar[r]^-{1\mu} & X\otimes A.}$$
Note that, since the action map $M\otimes A\to M$ is an epimorphism of right $A$-modules, any right $A$-module is a quotient of a free module.

More generally, for any right $A$-module $M$ and any $X\in\cC$ the tensor product $X\otimes M$ has a structure of a right $A$-module
$$\xymatrix{X\otimes M\otimes A \ar[r]^-{1\mu} & X\otimes M.}$$
This makes the category of modules $\cC_A$ a left module category over $\cC$. 
The adjoint pairing
$$\xymatrix{\cC_A \ar@/^10pt/[rr]^-{U} &\perp & \cC \ar@/^10pt/[ll]^-{-\otimes A} }$$
consisting of the forgetful and free $A$-module functors is an adjoint pair of $\cC$-module functors. 

The next notion provides algebras with semi-simple categories of modules.
An algebra $(A,\mu,\iota)$ in a rigid braided monoidal category $\cC$ is called {\em separable} if the following composition (denoted $e:A\otimes A\to I$) is a non-degenerate pairing:
$$\xymatrix{A\otimes A \ar[r]^(.6)\mu & A \ar[r]^{\epsilon} & I.}$$ 
Here $\epsilon$ is the composite
$$\xymatrix{A \ar[r]^-{1\ot\kappa_A} & A\otimes A\otimes A^* \ar[r]^-{\mu\ot 1} & A\otimes A^* \ar[r]^-{c_{A,A^*}} & A^*\otimes A \ar[r]^-{ev_A} & I,}$$ 
where $\kappa_A$ and $ev_A$ are duality morphisms for $A$. 
Non-degeneracy of $e$ means that there is a morphism $\kappa: I\to A\otimes A$ such that the composition
$$\xymatrix{A \ar[r]^-{1\ot\kappa} & A^{\otimes 3} \ar[r]^-{e\ot 1} & A}$$ 
is the identity. 
It also implies that the similar composition
$$\xymatrix{A \ar[r]^-{\kappa\ot 1} & A^{\otimes 3} \ar[r]^-{1\ot e} & A}$$ 
is also the identity.

Using the graphical calculus for morphisms in a (rigid) monoidal category \cite{js1} one can represent morphisms between tensor powers of a separable algebra by graphs (one dimensional CW-complexes), whose end vertices are separated into incoming and outgoing. 
For example, the multiplication map $\mu$ is represented by a trivalent graph with two incoming and one outgoing ends, the duality $\epsilon$ is an interval, with both incoming ends etc. 
It turns out (e.g in symmetric case, it follows from the results of \cite{rsw}) that separability implies that we can contract loops in connected graphs with at least one end.

For a separable algebra $A$ the adjunction
$$\xymatrix{\cC \ar@/^5pt/[r] & \cC_A \ar@/^5pt/[l] }$$
splits. 
Indeed, the splitting of the adjuction map $M\otimes A\to M$ is given by the projector $M\otimes A\to M\otimes A$:
$$\xymatrix{M\otimes A \ar[r]^-{I\ot\epsilon} & M\otimes A^{\otimes 3} \ar[r]^-{I\ot\mu\ot I} & M\otimes A^{\otimes 2} \ar[r]^-{\nu\ot I} & M\otimes A.}$$

For a separable algebra $A$ the effect on morphisms $\cC_A(M,N)\to\cC(M,N)$ of the forgetful functor $\cC_A\to\cC$ has a splitting $P:\cC(M,N)\to\cC_A(M,N)$:
$$P(f) = f,\quad f\in \cC_A(M,N).$$
For $f\in\cC(M,N)$ the image $P(f)$ is defined as the composition
$$\xymatrix{M \ar[r]^-{I\ot\epsilon} & M\otimes A^{\otimes 2} \ar[r]^-{\nu_M\ot I} & M\otimes A \ar[r]^-{f\ot I} & N\otimes A \ar[r]^-{\nu_N} & N.}$$
Moreover, the splitting has the properties 
$$P(fg)=fP(g)\ P(gh)=P(g)h,\quad f,h\in Mor\,(\cC_A),\ g\in Mor\,(\cC ).$$ 
This gives {\em Maschke's lemma} for separable algebras.

\begin{lemma}
Let $A$ be a separable algebra in a semi-simple rigid monoidal category $\cC$. Then the category $\cC_A$ of right
$A$-modules in $\cC$ is also semi-simple.
\end{lemma}

\subsection{Commutative algebras and local modules}\label{coal}

Now let $\cC$ be a braided monoidal category with the braiding $c_{X,Y}:X\otimes Y\to Y\otimes X$ (see \cite{js} for the definition). 
An algebra $A$ in $\cC$ is {\em commutative} if $\mu c_{A,A} = \mu$. 

It was shown in \cite{pa} that the category ${_A}{\cC}$ of left modules over a commutative algebra $A$ is monoidal with respect to the tensor product $M\otimes_AN$ over $A$, which can be defined by a coequaliser
$$\xymatrix{M\otimes_AN & M\otimes N \ar[l] && M\otimes A\otimes N \ar@/^5pt/[ll]^{(\nu_M1)(c_{M,A}1)} \ar@/_5pt/[ll]_{1\nu_N} }.$$

Moreover for commutative algebra $A$ the free functor $\cC\to\cC_A$ is (strong) monoidal which means the multiplication in $A$ induces an isomorphism
$$(X\otimes A)\otimes_A(Y\otimes A) \to (X\otimes Y)\otimes A.$$

A (right) module $(M,\nu)$ over a commutative algebra $A$ is {\em local} iff the diagram
$$\xymatrix{M\otimes A  \ar[r]^\nu \ar[d]_{c_{M,A}} & M\\ A\otimes M \ar[r]^{c_{A,M}} & M\otimes A \ar[u]_\nu}$$
commutes. Denote by $\cC_A^{loc}$ the full subcategory of $\cC_A$ consisting of local modules. The following result was
established in \cite{pa}.
\begin{proposition}
The category $\cC_A^{loc}$ is a full monoidal subcategory of $\cC_A$. Moreover, the braiding in $\cC$ induces a braiding in
$\cC_A^{loc}$.
\end{proposition}

The following statement was proved in \cite{ffrs}.
\begin{proposition}\label{algloc}
Let $(A,m,i)$ be a commutative algebra in a braided category $\cC$. Let $B=(B,\mu,\iota)$ be an algebra in ${\cC}_{A}$.
Define $\overline\mu$ and $\overline\iota$ as compositions $$\xymatrix{B\otimes B \ar[r] & B\otimes_AB \ar[r]^\mu & B,
& & & 1 \ar[r]^i & A \ar[r]^\iota & B.}$$ Then $\overline B=(B,\overline\mu,\overline\iota)$ is an algebra in $\cC$.
\newline
The map $\iota:A\to B$ is a homomorphism of algebras in $\cC$.
\newline
The algebra $\overline B$ in $\cC$ is separable or commutative if and only if the algebra
$B$ in ${\cC}_{A}$ is such.
\newline
The functor $(\cC_A^{loc})_B^{loc}\to \cC_{\overline B}^{loc}$
\begin{equation}\label{locloc}
(M,m:B\otimes_A M\to M)\mapsto (M,\overline m:B\otimes M\to B\otimes_A M\stackrel{m}{\to} M)
\end{equation}
is a braided monoidal equivalence.
\end{proposition}

A commutative algebra $A$ in a ribbon category $\cC$ (with the ribbon twist $\theta$)  is called {\em ribbon} if $\theta_A=1_A$. 

The next theorem is a part of theorem 4.5 from
\cite{KiO}.
\begin{theorem}\label{chiext}
Let $A$ be an indecomposable separable commutative, ribbon algebra in a modular category $\cC$. Then
$\cC_A^{loc}$ is a modular category. 
\end{theorem}
We call a commutative, separable, indecomposable, ribbon algebra $A\in\cC$ {\em trivialising} (or {\em Lagrangian} in the terminology of \cite{dmno}) if $\cC_A^{loc}$ is equivalent to the category $\Vect$ of vector spaces over the base field (that is, the only simple local $A$-module is $A$ itself). 
For any modular category $\cC$ the category $\cC\boxtimes\overline{\cC}$ always has a trivialising algebra $Z$, which we call the {\em diagonal} algebra (the {\em tube} algebra of \cite{mu}) with underlying object
$$Z = \oplus_{X\in Irr(\cC)}X\boxtimes X^*.$$

\subsection{Fusion rules and modular data}

A set $R$ is called a \emph{fusion rule} if its integer span $\mathbb{Z}R$ is equipped with a structure of an associative unital ring such that the unit element of  $\mathbb{Z}R$ belongs to $R$ and 
$$r\cdot s \in \mathbb{Z}_{\geq 0}R$$
for any $r,s\in R$.
Here $ \mathbb{Z}_{\geq 0}R$ is the sub ring in $\mathbb{Z}R$ of linear combinations of elements of $R$ with non-negative coefficients.
We also require  $\mathbb{Z}R$ to satisfy a \emph{rigidity condition}.
\newline
To formulate the rigidity condition equip $\mathbb{Z}R$ with a symmetric bilinear form $(-,-)$ defined by
$$(r,s) = \delta_{r,s}\quad r,s\in R.$$
Note that for  $x\in\mathbb{Z}_{\geq 0}R$ we have
$$(x,x)=1 \Longleftrightarrow x\in R.$$
The rigidity condition is then the existence of an involution $(-)^{\ast}:R\to R$ such that
$$(r\cdot s,t) = (s, r^{\ast}t)\quad r,s,t\in R.$$

By \emph{a homomorphism of fusion rules} $R\to S$ we mean a map of sets which induces a homomorphism of rings $\mathbb{Z}R\to\mathbb{Z}S$.

If $R$ and $S$ are fusion rules then so is $R\times S$.
In particular $R^{\times\ell} = R^{\ell}$ has the structure of a fusion rule.

Let $\bc$ be a semi-simple rigid monoidal category.
Then the set $Irr(\bc)$ of isomorphism classes of simple objects in $\bc$ has the structure of a fusion rule.
Note that $\mathbb{Z} Irr(\bc)$ coincides with the Grothendieck ring $K_{0}(\bc)$.

The {\em categorical dimension} of $\cC$ is $\mbox{Dim}(\cC) = \sum_{X}d(X)^2$, where the sum is taken over isomorphism classes of simple objects $X$ of $\cC$. 
Here $d(X) = tr(I_X)$ is the dimension of $X$. 

We use the definition of the {\em multiplicative central charge} of a modular category $\cC$ given in \cite[Section~ 6.2]{DGNO}
$$\xi(\cC) = \frac{1}{\sqrt{\mbox{dim}(\cC)}}\sum_{X}\theta(X)d(X)^2,$$
(the sum again is taken over simple objects) where $\theta_X = \theta(X)1_X$ is the twist on a simple object $X\in\cC$. 
We take the positive square root $\sqrt{\mbox{dim}(\cC)}$ of the positive real cyclotomic number $\mbox{dim}(\cC)$. 

The following properties are well known, see for example \cite[Section~ 3.1]{BaKi}.

\begin{lemma} \label{xi}
\begin{enumerate}
\item[(i)] $\xi(\bc)$ is a root of unity;
\item[(ii)] $\xi(\bc_1\boxtimes \bc_2)=\xi(\bc_1)\xi(\bc_2)$;
\item[(iii)] $\xi(\overline\bc)=\xi(\bc)^{-1}$. 
\end{enumerate}
\end{lemma}

Let $SL_2(\Z)$ be the \emph{modular group}; that is, the group of determinant 1 integer $2\times 2$-matrices. 
It is generated by the matrices
$$s = \left(\begin{array}{rr} 0 & -1\\ 1 & 0 \end{array}\right),\quad t = \left(\begin{array}{rr} 1 & 1\\ 0 & 1 \end{array}\right),$$
with the generating system of relations $s^4 = 1, (ts)^3 = s^2$.

Let $\cC$ be a modular category and define
$$S = (\sqrt{Dim(\cC)})^{-1}\tilde S,\quad T = \xi(\cC)^{-\frac{1}{3}}diag(\theta_X)$$ 
where $X$ runs through isomorphism classes of simple objects of $\cC$ and $\tilde S$ is the matrix defined in section (1.1).
The pair of matrices $S,T$ is often referred to as the {\em modular data} of $\cC$. 
The proof of the following result can be found in \cite{tu}.

\begin{theorem}
Let $\cC$ be a modular category. 
Then the operators $S$ and $T$ define an action of the modular group $SL_2(\Z)$ on the  complexified Grothendieck group $K_0(\cC)\otimes\C$.
\end{theorem}

The following is the second part of theorem 4.5 from
\cite{KiO}.
\begin{theorem}
The map $K_0(\cC_A^{loc})\otimes_\Z\C\to K_0(\cC)\otimes_\Z\C$, induced by the
forgetful functor $\cC_A^{loc}\to\cC$ is $SL_2(\Z)$-equivariant.
\end{theorem}

\subsection{NIM-representations}\label{nimrepsec}

Let $R$ be a fusion rule.

\noindent A set $M$ is a {\em non-negative integer matrix} (or NIM-)representation of $R$ if $\mathbb{Z}M$ is equipped with a structure of an  $\mathbb{Z}R$-module such that 
$$r\cdot m \in \mathbb{Z}_{\geq 0}M\quad \forall r \in R,\ m\in M$$
and $\mathbb{Z}M$ also possesses the rigidity condition.

As before, to formulate the rigidity condition observe that $\mathbb{Z}R$ comes equipped with a symmetric bilinear form 
$$(m,n) = \delta_{m,n};\, m,n\in M.$$
Again, it is obvious that for all $m\in\mathbb{Z}_{\geq 0}M$ we have
$$(m,m)=1 \Longleftrightarrow r\in M.$$

The rigidity condition for NIM-representations is then
$$(r\cdot m,n) = (m, r^{\ast}n)\quad r\in R,\, m,n\in M.$$
Note that a fusion rule $R$ is always a NIM-representation of itself.

More generally a homomorphism of fusion rules $R\to S$ turns $S$ into a NIM-representation of $R$.

A \emph{morphism of NIM-representations} $N\to M$ is a map of sets inducing a $\mathbb{Z}R$-module homomorphism $\mathbb{Z}N\to\mathbb{Z}M$.

The following is a complete reducibility statement for NIM-representations of a rigid fusion rule.

\begin{lemma}\label{cr}
Let $R$ be a fusion rule and let $N \subset M$ be an embedding of NIM-representations of $R$.
Then $$M\setminus N = \{m\in M|\quad m\not\in N\}$$ is a NIM-subrepresentation of $M$ and $M = N \sqcup (M\setminus N)$.
\end{lemma}
\begin{proof}
Note that $ \mathbb{Z}_{\geq 0}(M\setminus N)$ can be identified with the orthogonal complement of $\mathbb{Z}_{\geq 0}N$ in $\mathbb{Z}_{\geq 0}M$. Now for $m\in M\setminus N$, $r\in R$ and $n\in N$ we have
$$(r\cdot m, n) = (m, r^{\ast}n) = 0$$
since $r^{\ast}n\in\mathbb{Z}_{\geq 0}N$ and $(M\setminus N, N) = 0$.
Thus, $R\cdot(M\setminus N) \subset  \mathbb{Z}_{\geq 0}(M\setminus N)$. 
\end{proof}

Let $\bm$ be a semi-simple module category over a semi-simple rigid monoidal category $\bc$.
Then $Irr(\bm)$ is a NIM-representation of the fusion rule $Irr(\bc)$.

Note that the rigidity property for NIM-representations follows  from the adjunction
$$\bm (X\ast M, N) \simeq \bm (M, X^{\vee}\ast N).$$

\section{Fibonacci categories}\label{secfib}

In this section we describe modular categories with the {\em Fibonacci} fusion rule
$$\fibfr = \{1,x\} : x^2=1+x.$$
That is we classify all possible associativity constraints (F-matrices) and braidings (B-matrices). Although not written explicitly anywhere the results are known to specialists. We present them here for the sake of completeness. 

We consider a semi-simple $k$-linear category $\fib$ with simple objects $I$ and $X$.
The tensor product is defined by
$$X\otimes X = I \oplus X\, .$$
Under this definition there are two fundamental hom-spaces $\fib(X^{2},X)$ and $\fib(X^{2},I)$ for which the two respective basis vectors are

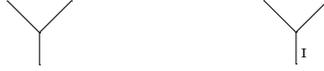
\begin{equation*}
\xygraph{ !{0;/r4.0pc/:;/u4.0pc/::}
[]*+{\TreeTwo} [rr]*+{\TreeTwoid}
 } 
\end{equation*}

\subsection{Associativity}\label{fibassoc}

The only non-trivial component of the associativity constraint for $\fib$ is $$\alpha_{X,X,X}:(X\otimes X)\otimes X\to X\otimes(X\otimes X).$$ 
On the level of hom-spaces this corresponds to two isomorphisms 
$$\fib(\alpha_{X,X,X},I):\fib(X\otimes(X\otimes X),I)\to\fib((X\otimes X)\otimes X,I)$$ 
and 
$$\fib(\alpha_{X,X,X},X):\fib(X\otimes(X\otimes X),X)\to\fib((X\otimes X)\otimes X,X).$$
Clearly $dim(\fib(X^{3},I))=1$ and $dim(\fib(X^{3},X))=2$ and so the associativity $k$-linear transformation for $\fib(X^{3},X)$ is given by  $A\in GL_2(k)$ and by $\alpha\in k^{\ast}$ for $\fib(X^{3},I)$.
Graphically, 

\begin{equation*}\def\tbmarg{6pt}
 \renewcommand{\objectstyle}{\labelstyle}
\xygraph{ !{0;/r4.0pc/:;/u4.0pc/::}
[]*+<0pt,\tbmarg>{\TreeThreeid}="ta" ( [rrr]*+<0pt,\tbmarg>{\TreeThreeidrev}="tb" !{+L*!R\txt\small{\phantom{aa}$\alpha$\phantom{aa}}."tb"="tb"}  )
("ta" :@{|->} "tb")
 } 
\end{equation*}
\begin{equation}\label{assco}
\begin{array}{ccc}
\left[\begin{array}{c}
\def\tbmarg{6pt}
 \renewcommand{\objectstyle}{\labelstyle}
\xygraph{ !{0;/r4.0pc/:;/u4.0pc/::}
[]*+<0pt,\tbmarg>{\TreeThreeA}
[d]*+<0pt,\tbmarg>{\TreeThreeB}
}
\end{array}\right]
& 
\xymatrix{\ar@{|->}[rr]&&\quad A}
&
\left[\begin{array}{c}
\def\tbmarg{6pt}
 \renewcommand{\objectstyle}{\labelstyle}
\xygraph{ !{0;/r4.0pc/:;/u4.0pc/::}
[]*+<0pt,\tbmarg>{\TreeThreeArev}
[d]*+<0pt,\tbmarg>{\TreeThreeBrev} 
}
\end{array}\right]
\end{array}
\end{equation}

\noindent where 
$$A = \left(\begin{array}{ccc} a & b \\ c & d \end{array}\right)\, .$$

The pentagon axiom of associativity coherence gives a set of equations (on $\alpha$ and matrix elements of $A$) for both $\fib(X^{4},I)$ and $\fib(X^{4},X)$.

The $k$-vector space $\fib(X^{4},I)$ is two-dimensional and so there are two groups of equations, one for each choice of a basis tree:

\begin{equation*}\def\tbmarg{6pt}
 \renewcommand{\objectstyle}{\labelstyle}
\xygraph{ !{0;/r4.0pc/:;/u4.0pc/::}
[]*+<0pt,\tbmarg>{\TreeFouridAtopa}="t1"(
:@{|->}[d]*+<0pt,\tbmarg>{\TreeFouridAbota}="t1a" (  !{+L*!R\txt\scriptsize{$a$}."t1a"="t1a"},
[d]*+<0pt,\tbmarg>{\TreeFouridAbotaid}="t1b" (:@{}[u]|{+})
)
)
[rr]*+<0pt,\tbmarg>{\TreeFouridAtopb}="t2" (!{+L*!R+\txt\scriptsize{$\alpha$}."t2"="tt2"} ("t1" :@{|->} "tt2") ,
[d]*+<0pt,\tbmarg>{\TreeFouridAbotb}="t2a" ( !{+L*!R\txt\scriptsize{$\phantom{a.}\alpha a$}."t2a"="tt2a"},
[d]*+<0pt,\tbmarg>{\TreeFouridAbotbid}="t2b" (:@{}[u]|{+}, !{+L*!R\txt\scriptsize{$b$}."t2b"="t2b"})
("t1a" :@{|->} "tt2a")
)
)
[rr]*+<0pt,\tbmarg>{\TreeFouridAtopc}="t3" (!{+L*!R+\txt\scriptsize{$\alpha^{2}$}."t3"="t3"} ("t2" :@{|->} "t3"), 
:@{=}[d]*+<0pt,\tbmarg>{\TreeFouridAtopc}="t3a" ( !{+L*!R\txt\scriptsize{$\phantom{a.}\alpha a^{2} + bc$}."t3a"="tt3a"},
[d]*+<0pt,\tbmarg>{\TreeFouridAtopcid}="t3b" (:@{}[u]|{+}, !{+L*!R\txt\scriptsize{$\alpha ab +bd$}."t3b"="t3b"})
("t2a" :@{|->} "tt3a")
)
)
}
\end{equation*}

\begin{equation*}\def\tbmarg{6pt}
 \renewcommand{\objectstyle}{\labelstyle}
\xygraph{ !{0;/r4.0pc/:;/u4.0pc/::}
[]*+<0pt,\tbmarg>{\TreeFouridAtopaid}="t1"(
:@{|->}[d]*+<0pt,\tbmarg>{\TreeFouridAbota}="t1a" (  !{+L*!R\txt\scriptsize{$c$}."t1a"="t1a"},
[d]*+<0pt,\tbmarg>{\TreeFouridAbotaid}="t1b" (:@{}[u]|{+}, !{+L*!R\txt\scriptsize{$d$}."t1b"="t1b"})
)
)
[rr]*+<0pt,\tbmarg>{\TreeFouridAtopbid}="t2" (!{+L*!R+\txt\scriptsize{$$}."t2"="tt2"} ("t1" :@{|->} "tt2") ,
[d]*+<0pt,\tbmarg>{\TreeFouridAbotb}="t2a" ( !{+L*!R\txt\scriptsize{$\phantom{a.}\alpha c$}."t2a"="t2a"},
[d]*+<0pt,\tbmarg>{\TreeFouridAbotbid}="t2b" (:@{}[u]|{+}="p2" , !{+L*!R\txt\scriptsize{$d$}."t2b"="t2b"})
("t1a" :@{|->} "t2a")
)
)
[rr]*+<0pt,\tbmarg>{\TreeFouridAtopcid}="t3" (!{+L*!R+\txt\scriptsize{$$}."t3"="t3"} ("t2" :@{|->} "t3"), 
:@{=}[d]*+<0pt,\tbmarg>{\TreeFouridAtopc}="t3a" ( !{+L*!R\txt\scriptsize{$\phantom{a.}\alpha ca + dc$}."t3a"="t3a"},
[d]*+<0pt,\tbmarg>{\TreeFouridAtopcid}="t3b" (:@{}[u]|{+}, !{+L*!R\txt\scriptsize{$\alpha cb + d^{2}$}."t3b"="t3b"})
("t2a" :@{|->} "t3a")
)
)
}
\end{equation*}

\noindent which yields four equations
\begin{eqnarray*}
\alpha a^{2} + bc &=& \alpha^{2} \\
\alpha ab + bd &=& 0 \\
\alpha cb + d^{2} &=& 1 \\
\alpha ca + dc &=& 0
\end{eqnarray*}

There are then three calculations for $\fib(X^{4},X)$ each corresponding to a different initial basis-tree

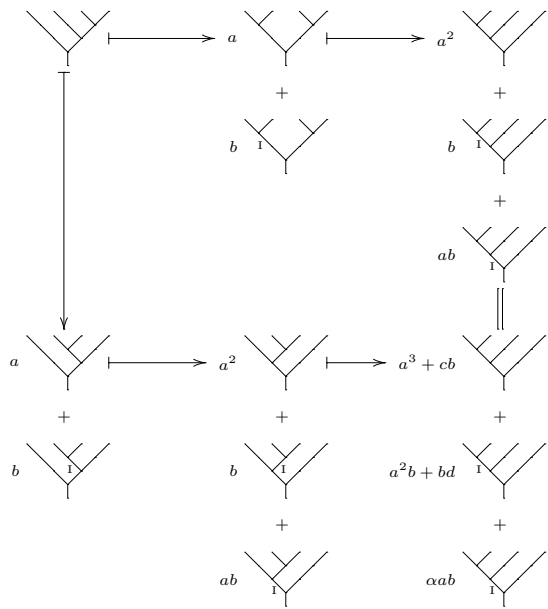
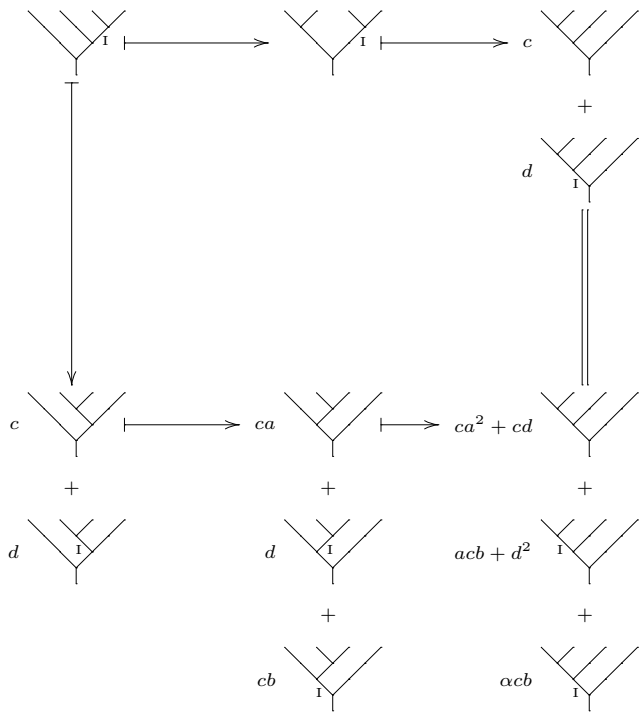
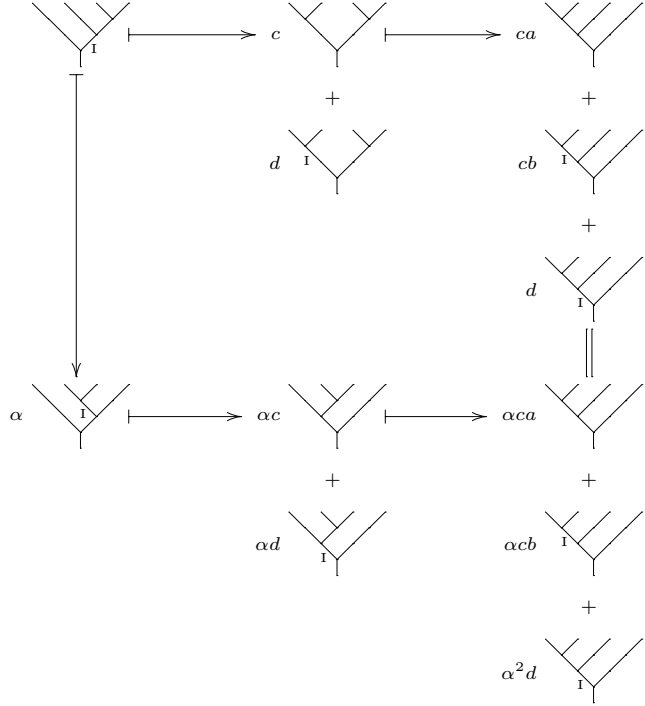
\begin{equation*}\scalebox{0.85}{\def\tbmarg{6pt}
 \renewcommand{\objectstyle}{\labelstyle}
\xygraph{ !{0;/r4.0pc/:;/u4.0pc/::}
[]*+<0pt,\tbmarg>{\TreeFouridBtopa}="t1a"( 
:@{|->}[ddd]*+<0pt,\tbmarg>{\TreeFouridBbota}="t1b" (!{+L*!R+\txt\scriptsize{$a$}})  ( 
[d]*+<0pt,\tbmarg>{\TreeFouridBbotaid}="t1c" (!{+L*!R+\txt\scriptsize{$b$}})  :@{}[u]|{+}
)
) 
[rr]*+<0pt,\tbmarg>{\TreeFouridBtopb}="t2a" (!{+L*!R+\txt\scriptsize{$a$}."t2a"="tt2a"} ("t1a" :@{|->} "tt2a") )   (
"t2a"[d]*+<0pt,\tbmarg>{\TreeFouridBtopbid}="t2b" (!{+L*!R+\txt\scriptsize{$b$}}) :@{}[u]|{+} ( 
"t2a"[ddd]*+<0pt,\tbmarg>{\TreeFouridBbotb}="t2c" (!{+L*!R+\txt\scriptsize{$a^{2}$}="t2cl"}) (
"t2a"[dddd]*+<0pt,\tbmarg>{\TreeFouridBbotbid}="t2d" (!{+L*!R+\txt\scriptsize{$b$}}) :@{}[u]|{+} !{"t2c"."t2cl"="tt2c"} ("t1b" :@{|->} "tt2c") (
"t2a"[ddddd]*+<0pt,\tbmarg>{\TreeFouridBbotbidid}="t2e" (!{+L*!R+\txt\scriptsize{$ab$}}) :@{}[u]|{+}
)
)
)
)
[rr]*+<0pt,\tbmarg>{\TreeFouridBtopc}="t3a" (!{+L*!R+\txt\scriptsize{$a^{2}$}."t3a"="tt3a"}) ("tt2a" :@{|->} "tt3a") 
"t3a"[d]*+<0pt,\tbmarg>{\TreeFouridBtopcid}="t3b" (!{+L*!R+\txt\scriptsize{$b$}}) :@{}[u]|{+}  
"t3a"[dd]*+<0pt,\tbmarg>{\TreeFouridBtopcidb}="t3c" (!{+L*!R+\txt\scriptsize{$ab$}}) :@{}[u]|{+}  
"t3a"[ddd]*+<0pt,\tbmarg>{\TreeFouridBtopc}="t3d" ("t3d" :@{=} "t3c") (!{+L*!R+\txt\scriptsize{$a^{3} + cb$}."t3d"="t3dl"}) ("tt2c" :@{|->} "t3dl") 
"t3a"[dddd]*+<0pt,\tbmarg>{\TreeFouridBtopcid}="t3e" (!{+L*!R+\txt\scriptsize{$a^{2}b +bd$}}) :@{}[u]|{+}
"t3a"[ddddd]*+<0pt,\tbmarg>{\TreeFouridBtopcidb}="t3f" (!{+L*!R+\txt\scriptsize{$\alpha ab$}}) :@{}[u]|{+}
}
}\end{equation*}

and,

\begin{equation*}\def\tbmarg{6pt}
 \renewcommand{\objectstyle}{\labelstyle}
\xygraph{ !{0;/r4.0pc/:;/u4.0pc/::}
[]*+<0pt,\tbmarg>{\TreeFouridBtopaid}="t1a"( 
:@{|->}[ddd]*+<0pt,\tbmarg>{\TreeFouridBbota}="t1b" (!{+L*!R+\txt\scriptsize{$c$}})  ( 
[d]*+<0pt,\tbmarg>{\TreeFouridBbotaid}="t1c" (!{+L*!R+\txt\scriptsize{$d$}})  :@{}[u]|{+}
)
) 
[rr]*+<0pt,\tbmarg>{\TreeFouridBtopbidid}="t2a" (!{+L*!R+\txt\scriptsize{$$}."t2a"="tt2a"} ("t1a" :@{|->} "tt2a") )   (
"t2a"[ddd]*+<0pt,\tbmarg>{\TreeFouridBbotb}="t2c" (!{+L*!R+\txt\scriptsize{$ca$}="t2cl"}) (
"t2a"[dddd]*+<0pt,\tbmarg>{\TreeFouridBbotbid}="t2d" (!{+L*!R+\txt\scriptsize{$d$}}) :@{}[u]|{+} !{"t2c"."t2cl"="tt2c"} ("t1b" :@{|->} "tt2c") (
"t2a"[ddddd]*+<0pt,\tbmarg>{\TreeFouridBbotbidid}="t2e" (!{+L*!R+\txt\scriptsize{$cb$}}) :@{}[u]|{+}
)
)
)
[rr]*+<0pt,\tbmarg>{\TreeFouridBtopc}="t3a" (!{+L*!R+\txt\scriptsize{$c$}."t3a"="tt3a"}) ("tt2a" :@{|->} "tt3a")  (
"t3a"[d]*+<0pt,\tbmarg>{\TreeFouridBtopcidb}="t3b" (!{+L*!R+\txt\scriptsize{$d$}}) :@{}[u]|{+}  ( 
"t3a"[ddd]*+<0pt,\tbmarg>{\TreeFouridBtopc}="t3d" ("t3d" :@{=} "t3b") (!{+L*!R+\txt\scriptsize{$ca^{2} + cd$}."t3d"="t3dl"}) ("tt2c" :@{|->} "t3dl") (
"t3a"[dddd]*+<0pt,\tbmarg>{\TreeFouridBtopcid}="t3e" (!{+L*!R+\txt\scriptsize{$acb + d^{2}$}}) :@{}[u]|{+} (
"t3a"[ddddd]*+<0pt,\tbmarg>{\TreeFouridBtopcidb}="t3f" (!{+L*!R+\txt\scriptsize{$\alpha cb$}}) :@{}[u]|{+}
)
)
)
)
}
\end{equation*}

and,

\begin{equation*}\def\tbmarg{6pt}
 \renewcommand{\objectstyle}{\labelstyle}
\xygraph{ !{0;/r4.0pc/:;/u4.0pc/::}
[]*+<0pt,\tbmarg>{\TreeFouridBtopaidid}="t1a"( 
:@{|->}[ddd]*+<0pt,\tbmarg>{\TreeFouridBbotaid}="t1b" (!{+L*!R+\txt\scriptsize{$\alpha$}}) 
) 
[rr]*+<0pt,\tbmarg>{\TreeFouridBtopb}="t2a" (!{+L*!R+\txt\scriptsize{$c$}."t2a"="tt2a"} ("t1a" :@{|->} "tt2a") )   (
"t2a"[d]*+<0pt,\tbmarg>{\TreeFouridBtopbid}="t2b" (!{+L*!R+\txt\scriptsize{$d$}}) :@{}[u]|{+} ( 
"t2a"[ddd]*+<0pt,\tbmarg>{\TreeFouridBbotb}="t2c" (!{+L*!R+\txt\scriptsize{$\alpha c$}="t2cl"}) (
"t2a"[dddd]*+<0pt,\tbmarg>{\TreeFouridBbotbidid}="t2d" (!{+L*!R+\txt\scriptsize{$\alpha d$}}) :@{}[u]|{+} !{"t2c"."t2cl"="tt2c"} ("t1b" :@{|->} "tt2c") 
)
)
)
[rr]*+<0pt,\tbmarg>{\TreeFouridBtopc}="t3a" (!{+L*!R+\txt\scriptsize{$ca$}."t3a"="tt3a"}) ("tt2a" :@{|->} "tt3a")  
"t3a"[d]*+<0pt,\tbmarg>{\TreeFouridBtopcid}="t3b" (!{+L*!R+\txt\scriptsize{$cb$}}) :@{}[u]|{+}  
"t3a"[dd]*+<0pt,\tbmarg>{\TreeFouridBtopcidb}="t3c" (!{+L*!R+\txt\scriptsize{$d$}}) :@{}[u]|{+}  
"t3a"[ddd]*+<0pt,\tbmarg>{\TreeFouridBtopc}="t3d" ("t3d" :@{=} "t3c") (!{+L*!R+\txt\scriptsize{$\alpha ca$}."t3d"="t3dl"}) ("tt2c" :@{|->} "t3dl") 
"t3a"[dddd]*+<0pt,\tbmarg>{\TreeFouridBtopcid}="t3e" (!{+L*!R+\txt\scriptsize{$\alpha cb$}}) :@{}[u]|{+} 
"t3a"[ddddd]*+<0pt,\tbmarg>{\TreeFouridBtopcidb}="t3f" (!{+L*!R+\txt\scriptsize{$\alpha^{2} d$}}) :@{}[u]|{+}
}
\end{equation*}

\noindent These diagrams provide a further eight equations

\begin{eqnarray*}
a^{3}+bc &=& a^{2} \\
a^{2}b+bd &=& b \\
ca^{2}+cd &=& c \\
abc+d^{2} &=& 0 \\
\alpha ab &=& ab \\
\alpha cb &=& d \\
\alpha ca &=& ca \\
\alpha^{2}d &=& cb.
\end{eqnarray*}



Solving all the equations proves the following lemma.






\begin{lemma}\label{assoclemma} The associativity constraint for the Fibonacci category is given by (\ref{assco}), where
$$\alpha = 1,\quad\quad{A=\left(\begin{array}{cc} a & b \\ -ab^{-1} & -a\end{array}\right)}$$
with $a$ being a solution of $a^2=a+1$. 
\end{lemma}
Note that $det(A)=-1$ and ${A^{-1}=A}$.

\subsection{Braiding}

The only non-trivial component of a braiding on $\fib$ is the isomorphism $c_{X,X}: X^{2}\to X^{2}$. 
On the level of hom-spaces $\fib (X^{2},1)$ and $\fib (X^{2},X)$ this isomorphism is given by
\begin{equation}\label{braid}\def\tbmarg{6pt}
 \renewcommand{\objectstyle}{\labelstyle}
\xygraph{ !{0;/r4.0pc/:;/u4.0pc/::}
[]*+{\TreeTwo}="t1" [rr]*+{\TreeTwo}="t2" (!{+L*!R+\txt\scriptsize{$u$}."t2"="t2"}) ("t1" :@{|->} "t2") 
[rr]*+{\TreeTwoid}="t3" [rr]*+{\TreeTwoid}="t4" (!{+L*!R+\txt\scriptsize{$w$}."t4"="t4"}) ("t3" :@{|->} "t4") ,
}
\end{equation}
where $w,u\in k^{\ast}$.

The coherence condition gives the following calculations

\begin{equation*}\def\tbmarg{6pt}
 \renewcommand{\objectstyle}{\labelstyle}
\xygraph{ !{0;/r4.0pc/:;/u4.0pc/::}
[]*++<0pt,\tbmarg>{\TreeThreeidrevs}="t1"
"t1"[urr]*+<0pt,\tbmarg>{\TreeThreeids}="t2" (!{+L*!R+\txt\scriptsize{$$}."t2"="t2"}) ("t1" :@{|->} "t2")
"t1"[urrrr]*+<0pt,\tbmarg>{\TreeThreeidrevs}="t3" (!{+L*!R+\txt\scriptsize{$w$}."t3"="t3"}) ("t2" :@{|->} "t3")
"t1"[urrrrrr]*+<0pt,\tbmarg>{\TreeThreeids}="t4" (!{+L*!R+\txt\scriptsize{$w$}."t4"="t4"}) ("t3" :@{|->} "t4") 
"t1"[drr]*+<0pt,\tbmarg>{\TreeThreeidrevs}="t5" (!{+L*!R+\txt\scriptsize{$u$}."t5"="t5"}) ("t1" :@{|->} "t5")
"t1"[drrrr]*+<0pt,\tbmarg>{\TreeThreeids}="t6" (!{+L*!R+\txt\scriptsize{$u$}."t6"="t6"}) ("t5" :@{|->} "t6")
"t1"[drrrrrr]*+<0pt,\tbmarg>{\TreeThreeids}="t7"  (!{+L*!R+\txt\scriptsize{$u^{2}$}."t7"="t7"}) ("t6" :@{|->} "t7")
("t4" :@{=} "t7")
}
\end{equation*}

\noindent and,

\begin{equation*}\def\tbmarg{6pt}
 \renewcommand{\objectstyle}{\labelstyle}
\xygraph{ !{0;/r4.0pc/:;/u4.0pc/::}
[]*++<0pt,\tbmarg>{\TreeThreeArevs}="t1"
"t1"[uurr]*+<0pt,\tbmarg>{\TreeThreeAs}="t2" (!{+L*!R+\txt\scriptsize{$a$}="tt2"})   (
[d]*+<0pt,\tbmarg>{\TreeThreeBs}="t2a" ("t2a" :@{}|{+} "t2") (!{"t2"."tt2"="t2"}) (!{+L*!R+\txt\scriptsize{$b$}."t2a"="t2a"}),
[ddd]*+<0pt,\tbmarg>{\TreeThreeArevs}="t2b" (!{+L*!R+\txt\scriptsize{$u$}."t2b"="t2b"}),
("t1" :@{|->} "t2") ("t1" :@{|->} "t2b") 
)  
"t1"[uurrrr]*+<0pt,\tbmarg>{\TreeThreeArevs}="t3" (!{+L*!R+\txt\scriptsize{$ua$}="tt3"})   (
[d]*+<0pt,\tbmarg>{\TreeThreeBrevs}="t3a" ("t3a" :@{}|{+} "t3") (!{"t3"."tt3"="t3"}) (!{+L*!R+\txt\scriptsize{$b$}."t3a"="t3a"}),
[ddd]*+<0pt,\tbmarg>{\TreeThreeAs}="t3b" (!{+L*!R+\txt\scriptsize{$ua$}="tt3b"}),
[dddd]*+<0pt,\tbmarg>{\TreeThreeBs}="t3c" ("t3c" :@{}|{+} "t3b") (!{"t3b"."tt3b"="t3b"}) (!{+L*!R+\txt\scriptsize{$ub$}."t3c"="t3c"})
("t2" :@{|->} "t3") ("t2b" :@{|->} "t3b")  
)  
"t1"[uurrrrrr]*+<0pt,\tbmarg>{\TreeThreeAs}="t4" (!{+L*!R+\txt\scriptsize{$ua^{2}+bc$}="tt4"})   (
[d]*+<0pt,\tbmarg>{\TreeThreeBs}="t4a" ("t4a" :@{}|{+} "t4") (!{"t4"."tt4"="t4"}) (!{+L*!R+\txt\scriptsize{$uab-ab$}="tt4a"}),
[ddd]*+<0pt,\tbmarg>{\TreeThreeAs}="t4b" ("t4a" :@{=} "t4b") (!{"t4a"."tt4a"="t4a"}) (!{+L*!R+\txt\scriptsize{$u^{2}a$}="tt4b"}),
[dddd]*+<0pt,\tbmarg>{\TreeThreeBs}="t4c" ("t4c" :@{}|{+} "t4b") (!{"t4b"."tt4b"="t4b"}) (!{+L*!R+\txt\scriptsize{$wub$}."t4c"="t4c"})
("t3" :@{|->} "t4") ("t3b" :@{|->} "t4b") 
)  
}
\end{equation*}

\noindent and,

\begin{equation*}\def\tbmarg{6pt}
 \renewcommand{\objectstyle}{\labelstyle}
\xygraph{ !{0;/r4.0pc/:;/u4.0pc/::}
[]*++<0pt,\tbmarg>{\TreeThreeBrevs}="t1"
"t1"[uurr]*+<0pt,\tbmarg>{\TreeThreeAs}="t2" (!{+L*!R+\txt\scriptsize{$c$}="tt2"})   (
[d]*+<0pt,\tbmarg>{\TreeThreeBs}="t2a" ("t2a" :@{}|{+} "t2") (!{"t2"."tt2"="t2"}) (!{+L*!R+\txt\scriptsize{$-a$}."t2a"="t2a"}),
[ddd]*+<0pt,\tbmarg>{\TreeThreeBrevs}="t2b" (!{+L*!R+\txt\scriptsize{$w$}."t2b"="t2b"}),
("t1" :@{|->} "t2") ("t1" :@{|->} "t2b") 
)  
"t1"[uurrrr]*+<0pt,\tbmarg>{\TreeThreeArevs}="t3" (!{+L*!R+\txt\scriptsize{$cu$}="tt3"})   (
[d]*+<0pt,\tbmarg>{\TreeThreeBrevs}="t3a" ("t3a" :@{}|{+} "t3") (!{"t3"."tt3"="t3"}) (!{+L*!R+\txt\scriptsize{$-a$}."t3a"="t3a"}),
[ddd]*+<0pt,\tbmarg>{\TreeThreeAs}="t3b" (!{+L*!R+\txt\scriptsize{$wc$}="tt3b"}),
[dddd]*+<0pt,\tbmarg>{\TreeThreeBs}="t3c" ("t3c" :@{}|{+} "t3b") (!{"t3b"."tt3b"="t3b"}) (!{+L*!R+\txt\scriptsize{$-wa$}."t3c"="t3c"})
("t2" :@{|->} "t3") ("t2b" :@{|->} "t3b")  
)  
"t1"[uurrrrrr]*+<0pt,\tbmarg>{\TreeThreeAs}="t4" (!{+L*!R+\txt\scriptsize{$cua-ac$}="tt4"})   (
[d]*+<0pt,\tbmarg>{\TreeThreeBs}="t4a" ("t4a" :@{}|{+} "t4") (!{"t4"."tt4"="t4"}) (!{+L*!R+\txt\scriptsize{$cub+a^{2}$}="tt4a"}),
[ddd]*+<0pt,\tbmarg>{\TreeThreeAs}="t4b" ("t4a" :@{=} "t4b") (!{"t4a"."tt4a"="t4a"}) (!{+L*!R+\txt\scriptsize{$uwc$}="tt4b"}),
[dddd]*+<0pt,\tbmarg>{\TreeThreeBs}="t4c" ("t4c" :@{}|{+} "t4b") (!{"t4b"."tt4b"="t4b"}) (!{+L*!R+\txt\scriptsize{$-w^{2}a$}."t4c"="t4c"})
("t3" :@{|->} "t4") ("t3b" :@{|->} "t4b") 
)  
}
\end{equation*}

\noindent We obtain the following equations

\begin{eqnarray*}
u^{2} &=& w  \\
u^{2}a &=& ua^{2} - a \\
wu &=& ua -a \\
-w^{2} &=& a- u
\end{eqnarray*}







\noindent Solving these proves the following lemma. 

\begin{lemma}\label{braidlemma} A braiding (\ref{braid}) on a Fibonacci category is completely determined by $u, w=u^2$ such that
$${u^{2} = ua -1}$$
\end{lemma}

\begin{remark} Note that $a = u + u^{-1}$ together with $a^{2} - a = 1$ implies that $u$ is a primitive root of unity of order $10$. 
Indeed, replacing $a$ by $u+u^{-1}$ in $a^2 = 1+a$ we get
$$u + u^{-1} +1 = (u + u^{-1})^2 = u^2 + 2 + u^{-2}$$ 
or 
$$0 = u^2 - u + 1 - u^{-1} + u^{-2} = u^{-2}(u^4 - u^3 + u^2 - u + 1).$$  
Thus the field of definition of a braided Fibonacci category is the cyclotomic field $\mathbb{Q}(\sqrt[10]{1})=\mathbb{Q}(\sqrt[5]{1})$.
\end{remark}



\subsection{Twist}

For both simple objects $X$ and $I$ the space of endomorphisms is one dimensional (generated by the identity morphism) and thus the twist automorphisms of $I$ and $X$ are simply scalar multiples of the identities on that object:
\begin{eqnarray*} \theta_{1}&=&\id_{I} \\  \theta_{X}&=&\rho\cdot\id_{X} \end{eqnarray*}
where $\rho\in k^{*}$ and naturality demands the scalar for $\theta_{I}$ to be $1$ (see \cite{tu}).

\noindent Using the coherence axiom for twists on the basis trees of $\fib (X^{2},X)$ and $\fib (X^{2},I)$ respectively gives

\begin{equation*}\def\tbmarg{6pt}
 \renewcommand{\objectstyle}{\labelstyle}
\xygraph{ !{0;/r4.0pc/:;/u4.0pc/::}
[]*++<0pt,\tbmarg>{\TreeTwo}="t1" (
[d]*+<0pt,\tbmarg>{\TreeTwo}="t2" (!{+L*!R+\txt\scriptsize{$u$}})  ("t1" :@{|->} "t2")
)
"t1"[rrd]*+<0pt,\tbmarg>{\TreeTwo}="t3" (!{+L*!R+\txt\scriptsize{$u^{2}$}="tt3"}) (!{"t3"."tt3"="t3"})  ("t2" :@{|->} "t3")
"t1"[rrrr]*+<0pt,\tbmarg>{\TreeTwo}="t4" (!{+L*!R+\txt\scriptsize{$\rho$}="tt4"}) (!{"t4"."tt4"="t4"})  ("t1" :@{|->} "t4") (
[d]*+<0pt,\tbmarg>{\TreeTwo}="t5" (!{+L*!R+\txt\scriptsize{$\rho^{2}u^{2}$}="tt5"}) (!{"t5"."tt5"="t5"}) ("t3" :@{|->} "t5")  ("t4" :@{=} "t5")
)
}
\end{equation*}

\noindent and

\begin{equation*}\def\tbmarg{6pt}
 \renewcommand{\objectstyle}{\labelstyle}
\xygraph{ !{0;/r4.0pc/:;/u4.0pc/::}
[]*++<0pt,\tbmarg>{\TreeTwoid}="t1" (
[d]*+<0pt,\tbmarg>{\TreeTwoid}="t2" (!{+L*!R+\txt\scriptsize{$u^{2}$}})  ("t1" :@{|->} "t2")
)
"t1"[rrd]*+<0pt,\tbmarg>{\TreeTwoid}="t3" (!{+L*!R+\txt\scriptsize{$u^{4}$}="tt3"}) (!{"t3"."tt3"="t3"})  ("t2" :@{|->} "t3")
"t1"[rrrr]*+<0pt,\tbmarg>{\TreeTwoid}="t4" (!{+L*!R+\txt\scriptsize{$1$}="tt4"}) (!{"t4"."tt4"="t4"})  ("t1" :@{|->} "t4") (
[d]*+<0pt,\tbmarg>{\TreeTwoid}="t5" (!{+L*!R+\txt\scriptsize{$u^{4}$}="tt5"}) (!{"t5"."tt5"="t5"}) ("t3" :@{|->} "t5")  ("t4" :@{=} "t5")
)
}
\end{equation*}

\begin{lemma} Twist structures on $\fib$ are completely determined by the braiding as
$${\rho = u^{-2}}.$$
\end{lemma}

\subsection{Monoidal Equivalences}

Let $\fib_{A}$ denote the monoidal category $\fib$ with the associativity matrix $A$ as first prescribed in \ref{fibassoc} with entries $a,b,c,d\in k^{\ast}$ and $\alpha = 1$. 
Suppose there is another associativity matrix for $\fib$ written
$$A'=\left(\begin{array}{cc} a' & b' \\ c' & d'\end{array}\right)$$
such that there is a monoidal equivalence $F: \fib_{A} \to \fib_{A'}$.
Due to the semi-simple nature of $\fib$ the only possibility for the underlying endo-functor of this equivalence is the identity (on objects).
Thus the monoidal functor is completely determined by the automorphisms on $\fib(X^{2},I)$ and $\fib(X^{2},X)$ given by

\begin{equation}\label{moneq}\def\tbmarg{6pt}
 \renewcommand{\objectstyle}{\labelstyle}
\xygraph{ !{0;/r4.0pc/:;/u4.0pc/::}
[]*++<0pt,\tbmarg>{\TreeTwoid}="t1" (
[d]*+<0pt,\tbmarg>{\TreeTwo}="t3"
)
"t1"[rr]*+<0pt,\tbmarg>{\TreeTwoid}="t2" (!{+L*!R+\txt\scriptsize{$f$}="tt2"}) (!{"t2"."tt2"="t2"})  ("t1" :@{|->} "t2") (
[d]*+<0pt,\tbmarg>{\TreeTwo}="t4" (!{+L*!R+\txt\scriptsize{$g$}="tt4"}) (!{"t4"."tt4"="t4"}) ("t3" :@{|->} "t4") 
)
}
\end{equation}

\noindent respectively, where $f,g\in k^{*}$.

\noindent The coherence condition for the monoidal equivalence gives the following calculations

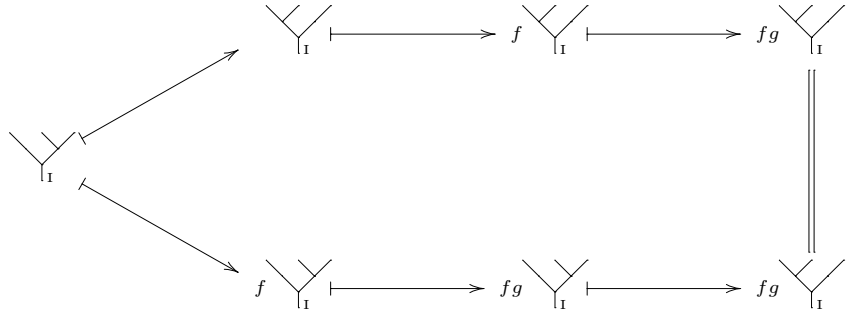
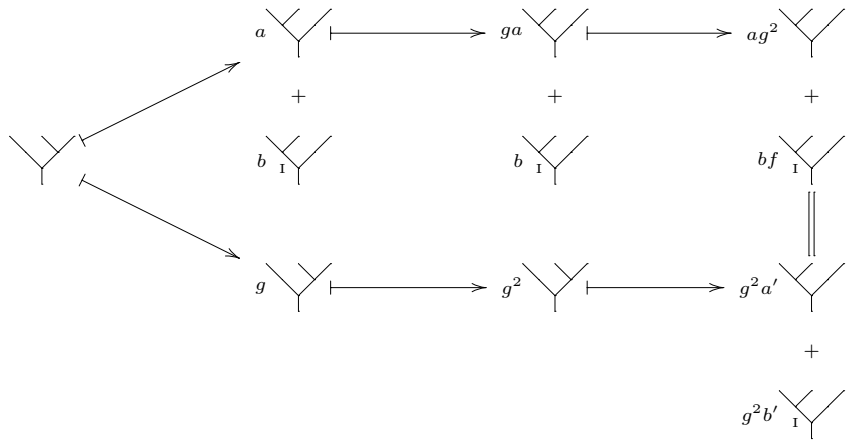
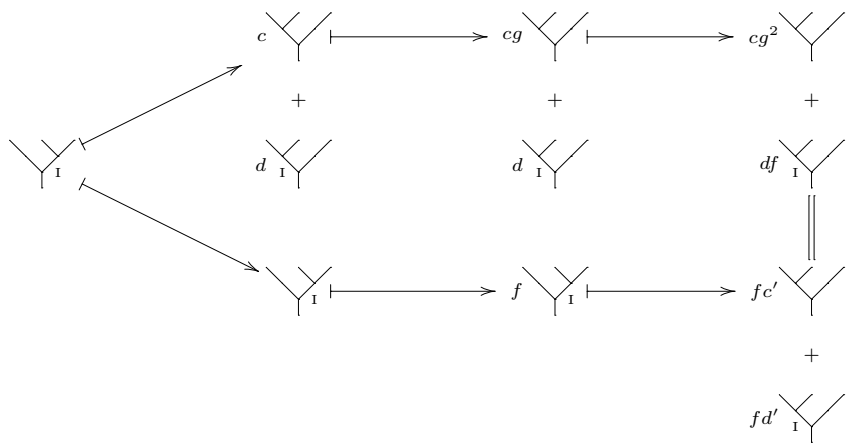
\begin{equation*}\def\tbmarg{6pt}
 \renewcommand{\objectstyle}{\labelstyle}
\xygraph{ !{0;/r4.0pc/:;/u4.0pc/::}
[]*++<0pt,\tbmarg>{\TreeThreeids}="t1" 
"t1"[urr]*+<0pt,\tbmarg>{\TreeThreeidrevs}="t2" (!{+L*!R+\txt\scriptsize{$\phantom{f}$}="tt2"}) (!{"tt2"."t2"="t2"}) ("t1" :@{|->} "t2") ( 
[dd]*+<0pt,\tbmarg>{\TreeThreeids}="t3" (!{+L*!R+\txt\scriptsize{$f$}="tt3"}) (!{"tt3"."t3"="t3"}) ("t1" :@{|->} "t3")
)
"t1"[urrrr]*+<0pt,\tbmarg>{\TreeThreeidrevs}="t4" (!{+L*!R+\txt\scriptsize{$f$}="tt4"}) (!{"tt4"."t4"="t4"}) ("t2" :@{|->} "t4") (
[dd]*+<0pt,\tbmarg>{\TreeThreeids}="t5" (!{+L*!R+\txt\scriptsize{$fg$}="tt5"}) (!{"tt5"."t5"="t5"}) ("t3" :@{|->} "t5")
)
"t1"[urrrrrr]*+<0pt,\tbmarg>{\TreeThreeidrevs}="t6" (!{+L*!R+\txt\scriptsize{$fg$}="tt6"}) (!{"t6"."tt6"="t6"}) ("t4" :@{|->} "t6") (
[dd]*+<0pt,\tbmarg>{\TreeThreeidrevs}="t7" (!{+L*!R+\txt\scriptsize{$fg$}="tt7"}) (!{"t7"."tt7"="t7"}) ("t5" :@{|->} "t7") ("t7" :@{=} "t6")
)
}
\end{equation*}

\noindent and

\begin{equation*}\def\tbmarg{6pt}
 \renewcommand{\objectstyle}{\labelstyle}
\xygraph{ !{0;/r4.0pc/:;/u4.0pc/::}
[]*++<0pt,\tbmarg>{\TreeThreeAs}="t1"
"t1"[urr]*+<0pt,\tbmarg>{\TreeThreeArevs}="t2" (!{+L*!R+\txt\scriptsize{$a$}="tt2"}) (!{"t2"."tt2"="t2"}) ("t1" :@{|->} "t2") ( 
[d]*+<0pt,\tbmarg>{\TreeThreeBrevs}="t3" (!{+L*!R+\txt\scriptsize{$b$}="tt3"}) (!{"t3"."tt3"="t3"}) ("t3" :@{}|{+} "t2")
[d]*+<0pt,\tbmarg>{\TreeThreeAs}="t4" (!{+L*!R+\txt\scriptsize{$g$}="tt4"}) (!{"t4"."tt4"="t4"}) ("t1" :@{|->} "t4")
)
"t1"[urrrr]*+<0pt,\tbmarg>{\TreeThreeArevs}="t5" (!{+L*!R+\txt\scriptsize{$ga$}="tt5"}) (!{"t5"."tt5"="t5"})  ("t2" :@{|->} "t5")  (
[d]*+<0pt,\tbmarg>{\TreeThreeBrevs}="t6"  (!{+L*!R+\txt\scriptsize{$b$}="tt6"}) (!{"t6"."tt6"="t6"})  ("t6" :@{}|{+} "t5")
[d]*+<0pt,\tbmarg>{\TreeThreeAs}="t7" (!{+L*!R+\txt\scriptsize{$g^{2}$}="tt7"}) (!{"t7"."tt7"="t7"}) ("t4" :@{|->} "t7")
)
"t1"[urrrrrr]*+<0pt,\tbmarg>{\TreeThreeArevs}="t8" (!{+L*!R+\txt\scriptsize{$ag^{2}$}="tt8"}) (!{"t8"."tt8"="t8"})  ("t5" :@{|->} "t8")  (
[d]*+<0pt,\tbmarg>{\TreeThreeBrevs}="t9" (!{+L*!R+\txt\scriptsize{$bf$}="tt9"}) (!{"t9"."tt9"="t9"})  ("t9" :@{}|{+} "t8")
[d]*+<0pt,\tbmarg>{\TreeThreeArevs}="t10" (!{+L*!R+\txt\scriptsize{$g^{2}a'$}="tt10"}) (!{"t10"."tt10"="t10"}) ("t7" :@{|->} "t10") ("t10" :@{=} "t9") 
[d]*+<0pt,\tbmarg>{\TreeThreeBrevs}="t11" (!{+L*!R+\txt\scriptsize{$g^{2}b'$}="tt11"}) (!{"t11"."tt11"="t11"}) ("t11" :@{}|{+} "t10")
)
}
\end{equation*}

\noindent and

\begin{equation*}\def\tbmarg{6pt}
 \renewcommand{\objectstyle}{\labelstyle}
\xygraph{ !{0;/r4.0pc/:;/u4.0pc/::}
[]*++<0pt,\tbmarg>{\TreeThreeBs}="t1"
"t1"[urr]*+<0pt,\tbmarg>{\TreeThreeArevs}="t2" (!{+L*!R+\txt\scriptsize{$c$}="tt2"}) (!{"t2"."tt2"="t2"}) ("t1" :@{|->} "t2") ( 
[d]*+<0pt,\tbmarg>{\TreeThreeBrevs}="t3" (!{+L*!R+\txt\scriptsize{$d$}="tt3"}) (!{"t3"."tt3"="t3"}) ("t3" :@{}|{+} "t2")
[d]*+<0pt,\tbmarg>{\TreeThreeBs}="t4" (!{+L*!R+\txt\scriptsize{$$}="tt4"}) (!{"t4"."tt4"="t4"}) ("t1" :@{|->} "t4")
)
"t1"[urrrr]*+<0pt,\tbmarg>{\TreeThreeArevs}="t5" (!{+L*!R+\txt\scriptsize{$cg$}="tt5"}) (!{"t5"."tt5"="t5"})  ("t2" :@{|->} "t5")  (
[d]*+<0pt,\tbmarg>{\TreeThreeBrevs}="t6"  (!{+L*!R+\txt\scriptsize{$d$}="tt6"}) (!{"t6"."tt6"="t6"})  ("t6" :@{}|{+} "t5")
[d]*+<0pt,\tbmarg>{\TreeThreeBs}="t7" (!{+L*!R+\txt\scriptsize{$f$}="tt7"}) (!{"t7"."tt7"="t7"}) ("t4" :@{|->} "t7")
)
"t1"[urrrrrr]*+<0pt,\tbmarg>{\TreeThreeArevs}="t8" (!{+L*!R+\txt\scriptsize{$cg^{2}$}="tt8"}) (!{"t8"."tt8"="t8"})  ("t5" :@{|->} "t8")  (
[d]*+<0pt,\tbmarg>{\TreeThreeBrevs}="t9" (!{+L*!R+\txt\scriptsize{$df$}="tt9"}) (!{"t9"."tt9"="t9"})  ("t9" :@{}|{+} "t8")
[d]*+<0pt,\tbmarg>{\TreeThreeArevs}="t10" (!{+L*!R+\txt\scriptsize{$fc'$}="tt10"}) (!{"t10"."tt10"="t10"}) ("t7" :@{|->} "t10") ("t10" :@{=} "t9") 
[d]*+<0pt,\tbmarg>{\TreeThreeBrevs}="t11" (!{+L*!R+\txt\scriptsize{$fd'$}="tt11"}) (!{"t11"."tt11"="t11"}) ("t11" :@{}|{+} "t10")
)
}
\end{equation*}

\noindent Put
$$G=\left(\begin{array}{cc} f & 0 \\ 0 & g^{2}\end{array}\right).$$
Monoidal coherence equations are equivalent to the matrix conjugation equation
$$A' = G^{-1} A G\,\, .$$

\begin{proposition}
Up to monoidal equivalence associativity constraints 
 \ref{assco} for Fibonacci category correspond to solutions of the equation $a^2 = a+1$: $$\alpha = 1,\quad {A=\left(\begin{array}{cc} a & 1 \\ -a & -a \end{array}\right)}.$$
For each associativity constraint on $\fib$ braided (balanced) structures up 
to braided equivalence correspond to solutions of $u^2 = au - 1$.
\end{proposition}
\begin{proof}
For any $A$ as found in lemma \ref{assoclemma} Choose $f$ and $g$ such that $f^{-1}g^{2} = b$ so that 
$$G^{-1}AG=\left(\begin{array}{cc} f^{-1} & 0 \\ 0 & g^{-2}\end{array}\right) \left(\begin{array}{cc} a & b \\ -ab^{-1} & -a \end{array}\right) \left(\begin{array}{cc} f & 0 \\ 0 & g^{2}\end{array}\right) = \left(\begin{array}{cc} a & 1 \\ -a & -a \end{array}\right)\, . $$
That there are exactly two braided structures is given by lemma \ref{braidlemma}. Clearly monoidal equivalences (\ref{moneq}) can not identify different braided (balanced) structures.
\end{proof}

\begin{remark}
We write $\fib_{u}$ for $\fib$ with a particular choice of parameterizing $u$ (and $a = u + u^{-1}$).
\end{remark}

\subsection{Duality and dimensions}

Here we show that Fibonacci categories are rigid; that is, any object has a dual, and that they are modular.

All we need to check is that $X$ has a dual. 
From the fusion rule it is clear that if $X^*$ exists it must be $X$. 
Thus, all we need to construct is the evaluation and coevaluation maps
$$ev:X\otimes X\to I,\quad coev:I\to X\otimes X,$$
such that the compositions
$$\xymatrix{X\ar[r]^-{1\otimes coev} & X\otimes(X\otimes X) \ar[rr]^-{\alpha_{X,X,X}^{-1}} && (X\otimes X)\otimes X \ar[r]^-{ev\otimes 1} & X}$$
\begin{equation}\label{dual}
\xymatrix{X\ar[r]^-{coev\otimes 1} & (X\otimes X)\otimes X \ar[rr]^-{\alpha_{X,X,X}} && X\otimes (X\otimes X) \ar[r]^-{1\otimes ev} & X}
\end{equation}
are identity. 
Since such  evaluation and coevaluation maps are unique up to a constant we can assume that $ev$ is the basic element in $\fib(X^2,I)$. 
To describe $coev$ note that the composition in $\fib$ gives a non-degenerate pairing 
$$\fib(I,X^2)\otimes\fib(X^2,I)\to \fib(I,I) = k.$$
We can choose a basic element in $\fib(I,X^2)$ to be the dual to the basic element in $\fib(X^2,I)$ as in \ref{fibassoc}. 
Then $coev$ is proportional to the basic element with the coefficient $\gamma$. 
The compositions (\ref{dual}) are both equal to $-a\gamma$ and so 
$$\gamma = -a^{-1} = 1-a.$$

\begin{lemma}
The dimension $dim(X)$ of $X$ is $1-a$.
\end{lemma}
\begin{proof}
The composition 
$$\xymatrix{I\ar[r]^-{coev} & X\otimes X \ar[rr]^-{(\theta_{X}\otimes 1)c_{X,X}} && X\otimes X \ar[r]^-{ev} & I}$$ 
coincides with $(1-a)1_X$.
\end{proof}

\begin{corollary}
The categorical dimension $Dim(\fib_{u})$ of the ribbon category $\fib_{u}$ is $3 - a$. 
The square of the multiplicative central charge $\xi(\fib_{u})^2$  is $u^{-1}$. 
\end{corollary}
\begin{proof}
For the categorical dimension
$$Dim(\fib_{u}) = dim(I)^2 + dim(X)^2 = 1 + (1-a)^2 = 1 + 1 - 2a + a^2 = 3 - a.$$
Note that, since $1 - u^{-1} + u^{-2} - u^{-3} + u^{-4}=0$,
$$\tau_+(\fib_{u}) = 1 + u^{-2}(1-a)^2 = 1 + u^{-2}(2-u-u^{-1}) = 1 + 2u^{-1} - u^{-1} - u^{-3}$$ 
coincides with $u^{-2}-u^{-4}$. 
Similarly, 
$$\tau_-(\fib_{u}) = 1 + u^2(1-a)^2 = 1 + u^2(2-u-u^{-1}) = 1 + 2u^{2} - u^{3} - u = u^2-u^4.$$
Thus for the square of the multiplicative central charge 
$$\xi(\fib_{u})^2 = \frac{\tau_+}{\tau_-} = \frac{u^{-4}(u^2-1)}{u^2(1-u^2)} = -u^{-6} = u^{-1}.$$
\end{proof}

\begin{proposition}
The ribbon category $\fib_{u}$ is modular. 
The $S$- and $T$-matrices are
$$S = \frac{1}{\sqrt{3-a}}\left(\begin{array}{cc}1& 1-a \\ 1-a & -1 \end{array}\right),\quad T =  u^\frac{1}{6}\left(\begin{array}{cc} 1 & 0 \\ 0 & u^2 \end{array}\right).$$
\end{proposition}
\begin{proof}
The first row (and column) of the $S$-matrix is filled with dimensions (in our case $1$ and $d$). 
The only non-trivial entry is the lower right corner, which is the trace of the square of the braiding on $X$. 
Since the braiding on $X$ has a form $(u^21_I,u1_X)$ with respect to the decomposition $X\otimes X = I\oplus X$. 
Its square has eigenvalues $(u^41_I,u^21_X)$. 
Then the trace $tr(c_{X,X}^2)$ can be written as 
$$u^4+u^2dim(X) = u^4+u^2(1-u-u^{-1}) = u^4-u^3+u^2-u = -1.$$ 
By the definition the matrix $T$ up to the factor $\xi(Fib_u)^{-\frac{1}{3}} = u^\frac{1}{6}$ is a diagonal matrix with diagonal entries being $\theta_1$ and $\theta_X$.
\end{proof}

The results of this section are summarized by the following theorem.

\begin{theorem}
Every braided balanced structure on a Fibonacci category is modular. 
Thus there are four non-equivalent Fibonacci modular categories $\fib_{u}$, parameterized by primitive roots of unity u of order 10.
\end{theorem}

Here is a simple consequence of the above theorem.
\begin{corollary}\label{decomp}
Let $\cC$ be a modular category and let $X\in\cC$ be such that $X^{\otimes 2} = I\oplus X$. Then
$$\cC\simeq \fib_u\boxtimes\caD$$ for some $u$ and a modular category $\caD$.
\end{corollary}
\begin{proof}
Since any braided balanced structure on a Fibonacci category is modular, $X$ generates a modular subcategory $\fib_u$  for some $u$.
Then by M\"uger's decomposition formula (\ref{mudec}) $$\cC\simeq \fib_u\boxtimes\cC_\cC(\fib_u).$$
\end{proof}
Note that it is the specific feature of the Fibonacci category that inforces the decomposition in corollary \ref{decomp}. For example an object $X$ in a modular category $\cC$ with the property $X^{\otimes 2}=I$ does not necessarily generates a modular subcategory in $\cC$.

\section{NIM-representations of $Fib^{\times\ell}$ and algebras in $\fib^{\boxtimes\ell}$}

In this section we study commutative, ribbon algebras in $\fib^{\boxtimes \ell}$. 
We do this by classifying NIM-representations of the Fibonacci fusion rule $Fib$ and its tensor powers $Fib^{\boxtimes \ell}$. 

We encode NIM-representations by certain types of oriented graphs. 
Nodes correspond to elements of a NIM-set $M$. 
Edges are colored in $\ell$ colours. 
Two nodes $m$ and $n$ are the source and the target of an $i$-th coloured edge respectively iff the multiplicity $(x*m,n)$ of $n$ in $x_i*m$ is non-zero. 
Here $x_i=1\otimes...\otimes 1\otimes X\otimes 1\otimes...\otimes 1$, where $X$ is in the $i$-th component.

\subsection{NIM-representation of $Fib^{\boxtimes \ell}$}

We begin by analyzing NIM-representations of $Fib$. Here we have only one colour for the edges.

\begin{lemma}\label{nimfib}
Up to isomorphism there is only one connected NIM-graph for the Fibonacci fusion rule
$$\xymatrix{m\ar@{-}[rr] && n \ar@(ur,dr) }\quad\,\,\, .$$
\end{lemma}
\begin{proof}
Let $m$ be a node of the graph of $M$. 
Write 
\begin{equation}\label{dec}
x*m = \sum_{i=1}^{k} m_{i}.
\end{equation}

\noindent An initial requisite observation is the following
$$x^{2}*m = x*(x*m) = x* \sum\limits_{i=1}^{k} m_{i} = k\cdot m + \sum\limits_{i=1}^{k} (x*m_{i} - m),$$
where $x*m_{i} - m$ is a non-negative linear combination of elements of the NIM-set $M$.

\noindent By the fusion rule $x^2=1+x$ we also have
$$x^{2}*m = x*m+ m$$
and thus
\begin{equation}x*m = (k-1)\cdot m + \sum\limits_{i=1}^{k} (xm_{i} - m) \end{equation}
Comparing this with (\ref{dec}) we see that all but one summand in (\ref{dec}) are equal to $m$ with the one remaining being say $n$. 
Thus
$$x*m = (k-1)m+n.$$
Now applying the fusion rule again we see that
$$k\cdot m +n = x*m+m = x^2*m = (k-1)x*m+x*n = (k-1)^2m+(k-1)n+x*n$$
or
$$(k-(k-1)^2)m+(2-k)n = x*n.$$
If $m=n$ then $x*m=(2-(k-1)^2)m$, which is in contradiction with the initial assumption $x*m=k\cdot m$.
Thus $m\not=n$ and $k$ is at most 2. 
We end up with two possibilities:
\newline
$k=1$ and $x*m=n$, or
\newline
$k=2$ and $x*m=m+n$. 

\noindent Note that in the first case 
$$x*n = x^{2}*m = (x+1)*m = xm + m = n + m.$$

\noindent While in the second case we have
$$x*n = x^{2}*m -x*m = (x+1)*m - x*m = m.$$

\noindent Thus up to the permutation $m\mapsto n$ we have only one possible indecomposable NIM-representation and one possible NIM-graph for $\fib$ fusion rule. 
\end{proof}

We now treat the general case of NIM-representations over tensor powers of $Fib$.

Let $M$ be a NIM-representation of $\fibfr^{\boxtimes \ell}$. 
It follows from lemma \ref{nimfib} that for any $m\in M$ and any $i=1,...,\ell$ the multiplicity $(x_{i}*m,m)$ can not be greater than 1. 
Define a map $\Gamma: M \to \{ 0 , 1 \}^\ell$  by $m \longmapsto \overline{m}$ where we write $\overline{m}_{i} = (x_{i}*m,m)$.
Let $<$ be the natural partial order on $\{ 0 , 1 \}^\ell$.

\begin{lemma}\label{gammalemma}
If nodes, corresponding to $n,m\in M$, are connected in the graph of the NIM representation $M$, then $\overline{m} < \overline{n}$ or $\overline{n} < \overline{m}$ in $\{ 0 , 1 \}^{\mathfrak{n}}$.
\end{lemma}
\begin{proof}
Suppose that $n$ and $m$ are connected by the $j$-th colored edge. 
Then we have $(x_{j}*n,n)=0$, $(x_{j}*n,m)=1$, and $(x_{j}*m,m)=(x_{j}*m,n)=1$.

\noindent In other words $x_{j}*m= m+n$ and so $n = x_{j}*m - m$.

\noindent We need to show that for any (other) $i$ that we have $\overline n_i\leq\overline m_i$. 
For each $i$ there are two possibilities, either $x_{i}*m = k$ ($\overline{m}_{i}=0$) or $x_{i}*m = k + m$ ($\overline{m}_{i}=1$). 
While it is obvious that $\overline n_i\leq\overline m_i$ in the second case, in the first case we have
$$x_{j}*k = (x_{j}x_{i})*m = (x_{i}x_{j})*m = x_{i}*n + x_{i}*m = x_{i}*n + k$$
\noindent and so $(x_{j}*k,x_{i}*n)=1$ which implies $(x_{i}*n,n)=0$ and $\overline n_i=\overline m_i$.
\end{proof}

\begin{remark} The image of $\Gamma$ is a sub-lattice in $\{ 0 , 1 \}^\ell$.\end{remark}

\begin{lemma}\label{embedlemma}
Let $\bm$ be a NIM-representation of $\fibfr^{k}$.
Let $m\in\bm$ such that $(x_{i}m,m) = (x_{i}x_{j} \ast m,m)=0$ for all $i,j = 1,...,k$.
Then the assignment $y\in\fibfr^{k} \mapsto y \ast m$ defines an embedding of NIM-representations $\fibfr^{k}\hookrightarrow\bm$.
\end{lemma}
\begin{proof}
We are required to show that  $y\in\fibfr^{k} \mapsto y \ast m$ is an embedding,  i.e. for $\epsilon ,\eta\in\{ 0,1 \}^{k}$
\begin{equation}\label{eqn39}(x_{1}^{\epsilon_{1}}...x_{k}^{\epsilon_{k}} \ast m,x_{1}^{\eta_{1}}...x_{k}^{\eta_{k}} \ast m) = \delta_{\epsilon_{1}\eta_{1}}...\delta_{\epsilon_{k}\eta_{k}}.\end{equation}
Assume (by induction) that this is true for any $\epsilon$ and $\eta$ with supports (sets of indexes of non-zero coordinates) in some proper subset of $[k] = \{ 1...k \}$.
We start by proving $(x_{1}...x_{k} \ast m,m) = 0$.

\noindent If $(x_{1}...x_{k} \ast m,m) \neq 0$ then by the assumption that $x_{1}...x_{k-1} \ast m$ and $x_{k} \ast m$ are simple we have
$$(x_{1}...x_{k-1} \ast m,x_{1}...x_{k-1} \ast m) = 1$$
and
$$(x_{k} \ast m,x_{k} \ast m) = 1$$
hence
$$(x_{1}...x_{k} \ast m,m) = (x_{1}...x_{k-1} \ast m,m)$$
should be equal to $1$, or $x_{1}...x_{k-1} \ast m = x_{k} \ast m$.
Similarly $x_{1} = x_{2}...x_{k}\ast m$.

\noindent But then 
\begin{eqnarray*}
x_{1}\ast m & = &  x_{2}...x_{k}\ast m \\
& = & (x_{2}...x_{k-1})\ast(x_{2}...x_{k-1}\ast m) \\
& = & x_{1}(x_{2} +1)...(x_{k-1}+1)\ast m \\
& = & x_{1}\ast m + x_{1}x_{2}\ast m + ... + x_{1}x_{k-1}\ast m + ...
\end{eqnarray*}
which contradicts, for example, that $(x_{1}x_{2}\ast m, x_{1}x_{2}\ast m)=1$.

\noindent Thus $(x_{1}^{\epsilon_{1}}...x_{k}^{\epsilon_{k}}\ast m , m) = 0$ for any non-zero $\epsilon\in \{ 0,1 \}^{k}$.
So we have
$$(x_{1}^{\epsilon_{1}}...x_{k}^{\epsilon_{k}} \ast m,x_{1}^{\eta_{1}}...x_{k}^{\eta_{k}} \ast m) = (m,m) \sum\limits_{i=1}^{k} \delta_{\epsilon_{i}\eta_{i}} + \textrm{ sum of } (x_{1}^{\xi_{1}}...x_{k}^{\xi_{k}}\ast m,m)$$
for non-zero $\xi_{1}...\xi_{k}$.
This proves (\ref{eqn39}).
That the image of $y\in\fibfr^{k} \mapsto y\ast m$ is a NIM-subrepresentation is obvious (see section \ref{nimrepsec}).
\end{proof}

\begin{theorem}\label{descriptiontheorem}
Any indecomposable NIM-representation of $\fibfr^{\ell}$ is of the form $\fibfr^{\lambda}$ for some set theoretic partition $\lambda$ of $[\ell] = \{ 1...\ell \}$.
\end{theorem}
\begin{proof}
Let $M$ be an indecomposable NIM-representation of $\fibfr^{\ell}$.
Our first step is to use lemma \ref{gammalemma} to show that there exists a $m\in M$ such that 
$$(x_{i} \ast m ,m) = 0\quad\forall i = 1,...,\ell.$$
Indeed, let $m$ be an element with a minimal (with respect to the partial order on $\{ 0,1 \}^{\ell}$) $\Gamma (m)$.
If $\Gamma (m)\neq 0$, then $\exists i$ such that $\Gamma (m)_{i} = (x_{i}\ast m,m)=1$ and then
$$((x_{i}-1)\ast m , (x_{i}-1)\ast m) = ((x_{i}-1)^{2}\ast m,m) = 2(m,m) - (x_{i}\ast m,m) = 1.$$
Thus $(x_{i}-1)\ast m\in M$.
By lemma \ref{gammalemma},  $\Gamma ((x_{i}-1)\ast m) < \Gamma (m)$ which contradicts the assumption.
Hence $\Gamma (m) = 0$ or $(x_{i}\ast m,m) = 0$ $\forall i = 1,...,\ell$.
This in particular implies that 
$$(x_{i}\ast m , x_{i}\ast m) = (x_{i}^{2}\ast m , m) = (x_{i}\ast m , m) + (m,m) = 1$$
and so $x_{i}\ast m \in M$.

Now define a set-theoretic partition $\lambda$ of $[\ell]$ by putting $i,j\in [\ell]$ in a given partition if and only if $x_{i}\ast m = x_{j}\ast m$.
Let $k$ be the number of parts of $\lambda$.
By permuting elements of $[\ell]$ we can assume that $x_{1},...,x_{k}$ lie in different parts of $\lambda$.
By lemma \ref{embedlemma} the assignment $y\mapsto y\ast m$ defines an injective map $\fibfr^{k} \hookrightarrow M$ of NIM-representations of $\fibfr^{k}$.
This obviously extends to the injective map $\fibfr^{\ell} \hookrightarrow M$ of NIM-representations of $\fibfr^{\ell}$.
By lemma \ref{cr} and the indecomposability of $M$, this map must be iso. 
\end{proof}

\begin{example}NIM-representations of $\fib^{\boxtimes 2}$
\end{example}
We have two colours in this case, which we depict by a solid line and a dashed line.
We have two set-theoretical partitions of $[2]$. The first $\{1\}\cup\{2\}=[2]$ corresponds to the square

$$
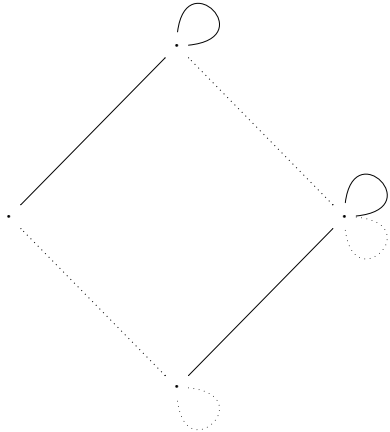
\xymatrix{
&& \cdot \ar@{.}[ddrr] \ar@{-}@(u,r) && \\
&&&& \\
\cdot \ar@{-}[uurr] \ar@{.}[ddrr] &&&& \cdot \ar@{-}@(u,r) \ar@{.}@(d,r) \\
&&&& \\
&& \cdot \ar@{-}[uurr] \ar@{.}@(d,r) &&
}
$$
\newline

while the second, $\{1,2\}=[2]$, corresponds to the double interval
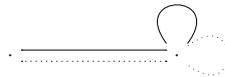
\newline

$$\renewcommand{\objectstyle}{\labelstyle}
\xymatrix{
\cdot \ar@{-}@<+0.5ex>[rr] \ar@<-0.5ex>@{.}[rr] && \cdot \ar@{-}@(ul,ur) \ar@{.}@(ur,dr)
}$$
\newline

\begin{example}NIM-representations of $\fib^{\boxtimes 3}$
\end{example}
We have three colours in this case, which we depict by a solid line, a broken line and a dashed line.
We have four set-theoretical partitions of $[3]$. 
An example is $\{1\}\cup\{2\}\cup\{3\}=[3]$ which corresponds to the cube
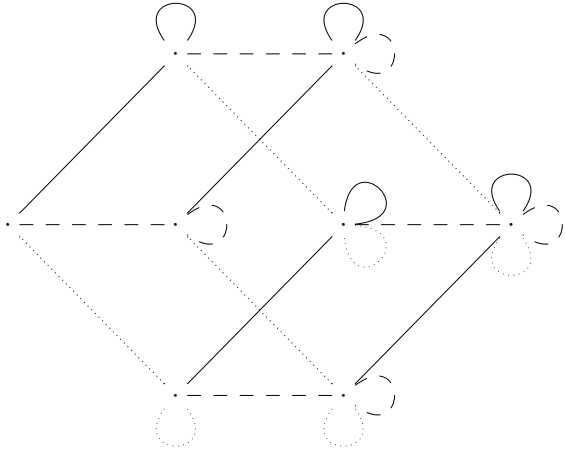
\newline

$$
\xymatrix{
&& \cdot \ar@{.}[ddrr] \ar@{-}@(ul,ur) \ar@{--}[rr] && \cdot \ar@{.}[ddrr] \ar@{-}@(ul,ur) \ar@{--}@(ur,dr) &&&  \\
&&&& & \\
\cdot \ar@{-}[uurr] \ar@{.}[ddrr] \ar@{--}[rr] && \cdot \ar@{-}[uurr] \ar@{.}[ddrr] \ar@{--}@(ur,dr) && \cdot \ar@{-}@(u,r) \ar@{.}@(d,r) \ar@{--}[rr] && \cdot \ar@{-}@(ul,ur) \ar@{--}@(ur,dr) \ar@{.}@(dl,dr) & \\
&&&&& \\
&& \cdot \ar@{-}[uurr] \ar@{.}@(dl,dr) \ar@{--}[rr] && \cdot \ar@{-}[uurr] \ar@{--}@(ur,dr) \ar@{.}@(dl,dr) &&&
}
$$
\newline
\newline
\newline

Although we do not classify all possible module categories of $\fib^{\boxtimes\ell}$, the description of their NIM-representations (obtained in theorem \ref{descriptiontheorem}) is enough to prove that there are no non-trivial ribbon algebras in $\fib^{\boxtimes\ell}$ which we establish in the next section.

\subsection{Commutative algebras in $\fib^{\boxtimes\ell}$}
Here we look at commutative, ribbon algebras in products of Fibonacci modular categories.

\begin{theorem}\label{noalg}
There are no ribbon algebras in $\fib_u^{\boxtimes\ell}$.
\end{theorem}
\begin{proof}
Let $A$ be an indecomposable algebra in $\bc = \fib_u^{\boxtimes\ell}$.
Then $\bc_{A}$ is an indecomposable $\bc$-module category (see section \ref{algandmods}) of $A$-modules in $\bc$.
As was noted in section \ref{algandmods} the forgetful functor $F:\bc_{A}\to\bc$ (forgetting the module structure) is a functor of module categories over $\bc$ and has the right adjoint $G:\bc\to\bc_{A}$ which is again a $\bc$-module functor.
Note that the adjoint sends the monoidal unit $I$ to $A$ as a module over itself.
Hence for the NIM-representation $M$ of $\bc_{A}$ we have maps of NIM-representations 
$$f:M\to\fibfr^{\boxtimes\ell}\,\, \textrm{ and }\,\, g:\fibfr^{\boxtimes\ell}\to M$$
which are adjoint, i.e.
$$(g(y),m)_{M} = (g,f(m))_{\fibfr^{\boxtimes\ell}}.$$

\noindent Since $\bc_{A}$ is indecomposable as a $\bc$-module category, so is its fusion rule $M$.
According to theorem \ref{descriptiontheorem} we should have $M \simeq \fibfr^{\lambda}$ for some set-theoretic partition $\lambda$ of $[\ell]$.

\noindent Assume that $\lambda$ has only one part $\lambda = (\ell)$.
In particular $M=\fibfr^{(\ell)}$ has just two simple objects: $m$ and $n$.
Assume that $m=g(1)$.
Since $g$ is a map of NIM-representations we have $g(x_{i})=n$ for all $i=1,...,\ell$ such that 
$$g(x_{i}\ast 1) = x_{i} g(1) = x_{i}\ast m =n.$$
Hence for an arbitrary element $x_{i_{1}}...x_{i_{s}}$ of $\fibfr^{\boxtimes\ell}$
$$g(x_{i_{1}}...x_{i_{s}}\ast 1) = f_{s}\ast n + f_{s-1}\ast m$$
where $f_{s}$ is the $s$-th Fibonacci number.
Thus the adjoint map $f$ has the form
\begin{eqnarray}\label{f}
f(m) & = & 1 + \sum\limits_{s=1}^{\ell} f_{s-1}\ast\sum\limits_{i_{1}<...<i_{s}}x_{i_{1}}...x_{i_{s}} \\
f(n) & = & \sum\limits_{s=1}^{\ell} f_{s}\ast \sum\limits_{i_{1}<...<i_{s}}x_{i_{1}}...x_{i_{s}}
\end{eqnarray}
Obviously $f(m)$, which is the class of the algebra $A$ in $K_{0}(\fib_u^{\boxtimes\ell})=\mathbb{Z}[\fibfr^{\boxtimes\ell}]$, cannot be ribbon.

\noindent Assume now that $\lambda$ is a set-theoretic partition of $[\ell]$ into ordered parts
$$\lambda = [1...\ell_{1}][\ell_{1}...\ell_{2}]...[\ell_{n-1}...\ell_{n}].$$
In this case according to theorem \ref{descriptiontheorem}, $M$ as a NIM-representation of $\fib_u^{\boxtimes\ell} = \fib_u^{\boxtimes\ell_{1}}\boxtimes ...\boxtimes\fib_u^{\boxtimes\ell_{s}}$ has the form
$$M = \fibfr^{(\ell_{1})}\boxtimes ... \boxtimes \fibfr^{(\ell_{s})}$$
By the above $\mathbb{Z}[\fibfr^{(\ell_{j})}] = K_0((\fib_u^{\boxtimes\ell_{i}})_{A})$ for some non-ribbon algebra $A_{i}\in\fib_u^{\boxtimes\ell_{i}}$.
Then $\mathbb{Z}M = K_0(\boxtimes^{s}_{j=1} (\fib_u^{\boxtimes\ell_{i}})_{A_{i}}) = K_0((\fib_u^{\boxtimes\ell})_{A})$ where $A=\boxtimes^{s}_{j=1}A_{i}$.
Since $A_{i}$ are non-ribbon then so is $A$.
\newline
The case for general $\lambda$ can be reduced to the above by a permutation of $[n]$.
\end{proof}
\begin{remark}
Note that the product $\fib_u\boxtimes\fib_{u^{-1}}$ has a commutative ribbon algebra, whose class in $K_0(\fib_u\boxtimes\fib_{u^{-1}}) = Fib^{\otimes 2}$ is $1 + x_1x_2$ (see section \ref{coal}). The corresponding NIM-representation is $Fib^{(2)}$. 
\newline
At the same time the argument of the proof of theorem \ref{noalg} works well for $\fib_u^{\boxtimes\ell}\boxtimes\fib_v^{\boxtimes m}$ as long as $uv\not= 1$, thus proving that there are no commutative ribbon algebras in $\fib_u^{\boxtimes\ell}\boxtimes\fib_v^{\boxtimes m}$ for such $u,v$. 
\end{remark}

\section{Applications}

Here we look at vertex operator algebras (=chiral algebras) whose categories of modules are tensor products of Fibonacci categories.

\subsection{Rational vertex operator algebras} \label{vex}

We start by recalling basic facts about rational vertex operator algebras and their representations. 
 
Let $\V$ be a rational vertex operator algebra (or VOA for short) of central charge $c$; that is, a vertex algebra satisfying conditions 1-3 from \cite[Section~ 1]{Hu}. 

To each $\V$-module $\M$ one can associate its {\em character} 
$$\chi_\M(q) = tr_\M(q^{L_0-\frac{c}{24}}).$$
Here $tr_\M$ is the trace of an operator on a vector space $\M$, $L_0$ is the zero Virasoro mode and $c$ is the central charge of $\V$. 
The action of $L_0$ on $\M$ is diagonalisable so that
$$\M = \bigoplus_{l}\M_l,\quad\quad \M_l = \{m\in\M,\ L_0(m) = lm\}.$$
Moreover if $\M$ is irreducible then 
$$\M = \bigoplus_{n\geq0}\M_{h+n}$$ 
for some rational number $h=h_M$, the {\em conformal weight} of $\M$. 
Thus the character of an irreducible $\M$ can be written as
$$\chi_\M(q) = q^{h-\frac{c}{24}}\sum_{n\geq 0}dim(\M_{h+n})q^n.$$

The (complex) span of characters of a rational VOA is a module over the modular group $SL_2(\mathbb{Z})$ with respect to the action
$$S(\chi)(\tau) = \chi(-\frac{1}{\tau}),\quad\quad T(\chi) = \chi(\tau+1).$$
Here we make the change of variable $q = e^{2\pi i\tau}$. 

It is proved in \cite{Hu} that the category $\Rep(V)$ of $V$-modules of finite length has the natural structure of a modular tensor category. 
We will denote the fusion of $V-$modules $M$ and $N$ by $M*N$.

The relation between the central charge $c_V$ of a unitary rational VOA $V$ and the central charge of the category of its modules $\Rep(V)$ is given by (e.g. see \cite{ms,re}): $$\xi(\Rep(V)) = e^{\frac{2\pi ic_V}{8}}.$$ In the non-unitary case only the square of the above identity is true. The ribbon twist on an irreducible $V$-module $M$ is related to its conformal weight $h_M$ as follows (see \cite{Hu}):  
$$\theta_M = e^{2\pi ih_M}I_M.$$

Note that a rational vertex algebra has to be {\em simple} (i.e. has no non-trivial ideals). 
This, in particular, means that VOA maps between rational vertex algebras are monomorphisms.

The category of modules $\Rep(\V\otimes \U)$ of the tensor product of two (rational) vertex algebras is ribbon equivalent to the tensor product $\Rep(\V)\boxtimes\Rep(\U)$ of the categories of modules (see, for example \cite{fhl}). Sometimes following physics tradition we will write $\V\times \U$ for the tensor product  $\V\otimes \U$ of VOAs. For $\V$-module $M$ and $\U$-module $N$ the $\V\otimes \U$-module with underlying vector space $M\otimes N$ will be denoted by $M\boxtimes N$.

Now consider an extension of vertex operator algebras $V\subset W$, where $V$ is a rational vertex operator algebra of $W$. Then $W$ viewed as a $V$-module decomposes into a finite direct sum of irreducible $V$-modules.
Moreover $W$ considered as an object $A\in \Rep(V)$ has a natural structure of simple, separable, commutative, ribbon algebra, see \cite[Theorem 5.2]{KiO}.
The converse is also true if the conformal weights of the irreducible components of the underlying $V$-representation of the algebra are all positive. In particular it is always true if $V$ is a unitary VOA. 

\medskip
{\bf Holomorphic VOAs}
\newline
A VOA is called {\em holomorphic} if it has only one irreducible module, namely itself. It is known (see e.g. \cite{sch,dlm}) that the central charge of a holomorphic VOA is divisible by 8. 

\medskip
{\bf Diagonal extensions}
\newline
Suppose that $\V$ is a VOA whose category of representations has a form $\cC\boxtimes\overline\cC$ for some modular category $\cC$. The diagonal algebra $Z\in \cC\boxtimes\overline\cC$ (from section \ref{coal}) defines a holomorphic extension of $\V$, which is automatically a VOA in the unitary case. 

\medskip
{\bf Simple current extensions}
\newline
Recall that a module over a VOA is called {\em simple current} if it is invertible.
Let $S$ be a subgroup of the group of simple currents of a rational VOA such that conformal weights of elements of $S$ are positive integer (in the non-unitary case we should also assume that the braiding $c_{s,s}=1$ for any $s\in S$). The object $\oplus_{s\in S}s$ of the category $\Rep(\V)$ has a structure of simple, separable, commutative, ribbon algebra  and the corresponding VOA extension of $\V$ is called the {\em simple current} extension (see e.g. \cite{fss}).

\medskip
{\bf Conformal embeddings}
\newline
Let $\bigoplus_i\g^i \subset \g'$ be an embedding (here $\g^i$ and $\g'$ are finite dimensional simple Lie algebras). 
We will symbolically write $\oplus_i(\g^i)_{k_i}\subset \g'_{k'}$ if the restriction of a $\hat \g'-$module of level $k'$ to $\hat \g^i$ has level $k_i$ (in this case the numbers $k_i$ are multiples of $k'$). 
Such an embedding defines an embedding of vertex algebras $\otimes_iV(\g^i,k_i)\subset V(\g',k')$; but in general this embedding does not preserve the Virasoro element. 
In the case when it does, the embedding $\oplus_i(\g^i)_{k_i}\subset \g'_{k'}$ is called {\em conformal}, see \cite{bb, SW, KW}.

\medskip
{\bf Cosets}
\newline
Let $\U\subseteq \V$ be an embedding of rational vertex algebras which does not preserve conformal vectors $\omega_\U,\omega_\V$ (only operator products are preserved). 
The {\em centralizer} $C_\V(\U)$ is a vertex algebra with the conformal vector $\omega_\V-\omega_\U$ (see \cite{fz}). 
Note that the tensor product $\U\otimes C_\V(\U)$ is mapped naturally to $\V$ and this map is a map of vertex operator algebras. 
In the case when $\V,\U$ and $C_\V(\U)$ are rational this map is an embedding (by simplicity of $\U\otimes C_\V(\U)$) and $\V$ is a commutative algebra in the category of modules
$$\Rep(\U)\boxtimes\Rep(C_\V(\U))\simeq \Rep(\U\otimes C_\V(\U)).$$ 

The coset construction allows us to describe VOAs with a given completely anisotropic category of representations.
\begin{proposition}
Let $\V$ be a rational VOA with a completely anisotropic category of representations $Rep(\V)$. Then any rational VOA whose category of representations is $\overline{Rep(\V)}$ is a coset $\U/\V$ of some holomorphic $\U$. 
\end{proposition}
\begin{proof}
Let $\W$ be a VOA such that $Rep(\W) = \overline{Rep(\V)}$. Then $$Rep(\V\otimes \W) = Rep(\V)\boxtimes\overline{Rep(\V)}$$ contains a commutative ribbon separable indecomposable algebra with trivial category of local modules. Thus the corresponding VOA extension $\U$ of $\V\otimes\W$ is holomorphic. The coset $\U/\V$ is an extension of $\W$ and by complete anisotropy of $\Rep(\W)$ must coincide with $\W$. 
\end{proof}

\medskip
{\bf Lattice VOAs} 
\newline
Let $L$ be a positive definite integral even lattice of rank d, $(\ ,\ )$ the pairing. Its 
associated vertex operator algebra $\V_L$ is rational and unitary and has central charge d. Its irreducible representations 
are in one-to-one correspondence with the cosets in $L^\# /L$, where $L^\#$ is the dual lattice 
of $L$ \cite{do}.  
Fusion rule are given by the addition in the finite abelian group $L^\# /L$.

\medskip
{\bf Affine VOAs} 
\newline
Let $\g$ be a finite dimensional simple Lie algebra and let $\hat \g$ be the corresponding affine Lie algebra. For a positive integer $k$ let  $\g_k = V(\g,k)$ be the simple vertex operator algebra associated with the vacuum $\hat \g-$module of level $k$ \cite{fz}. This VOA is rational of central charge 
\begin{equation} \label{cformula}
c(\g_k)=\frac{k\dim \g}{k+h^\vee},
\end{equation}
where $h^\vee$ is the dual Coxeter number of the Lie algebra $\g$.

In particular the central charge of $A_{1,k}$ is $\frac{3k}{k+2}$. 
The irreducible $V(A_1,k)$-modules are irreducible highest weight modules $L_{A_{1,k}}(i)$ with $i=0,...,k$.
The conformal weight of $L_{A_1}(k,i)$ is $\frac{i(i+2)}{4(k+2)}$. 
The fusion rule is given by
$$L_{A_{1,k}}(i)*L_{A_{1,k}}(j) = \left\{\begin{array}{cc}\bigoplus\displaylimits^{min(i,j)}_{s=0} L_{A_{1,k}}(i+j-2s), & i+j<k\\ \\ \bigoplus\displaylimits^{min(i,j)}_{s=i+j-k} L_{A_{1,k}}(i+j-2s), & i+j\geq k\end{array}\right.$$
In particular $L_{A_{1,k}}(k)$ is invertible (a {\em simple current}): 
$$L_{A_{1,k}}(k)*L_{A_{1,k}}(k) = L_{A_{1,k}}(0).$$ Its conformal weight is $k/4$. 
Its action on other irreducible modules is 
$$L_{A_{1,k}}(k)*L_{A_{1,k}}(i) = L_{A_{1,k}}(k-i).$$
\newline

\medskip
{\bf Minimal models}
\newline
Let $1<p<q$ be a coprime integers. Following  \cite{bpz,fms,wa} consider the minimal (Virasoro) VOA $M(p,q)=\Vir_{c_{p,q}}$ of 
central charge 
$$c_{p,q} = 1-6\frac{(p-q)^2}{pq}.$$
Irreducible representations of $M(p,q)$ are irreducible Virasoro highest weight modules $L_{Vir_{c_{p,q}}}(h_{(r,s)})=L_{p,q}(r,s)$ with conformal weights 
$$h_{(r,s)} = \frac{(qr-ps)^2-(p-q)^2}{4pq}$$
labelled by 
$$(r,s),\quad 1\leq r\leq p-1,\quad 1\leq s\leq q-1$$ modulo identification $(p-r,q-s) = (r,s)$.

Fusion rules have the following compact presentation:
$$L_{p,q}(r,s)L_{p,q}(r',s') = \osum\displaylimits_{i=1+|r-r'|}^{min(r+r'-1,2p-r-r'-1)}\ \ \osum\displaylimits_{j=1+|s-s'|}^{min(s+s'-1,2q-s-s'-1)}\ L_{p,q}(i,j),$$ where the primed summation indicates that $i$ and $j$ increment in twos.

Note that for $p>2$ the minimal model $M(p,q)$ has a simple current $L_{p,q}(p-1,1)$ of conformal weight $\frac{(p-2)(q-2)}{4}$. Its fusion is
$$L_{p,q}(p-1,1)L_{p,q}(r,s) = L_{p,q}(p-r,s).$$ In particular $L_{p,q}(p-1,1)$ generates a cyclic group of order 2.

\subsection{Fibonacci chiral algebras}

Here we look at examples where the Fibonacci categories are categories of representations of rational vertex operator algebras (VOAs of {\em Fibonacci} type). In particular we present all four modular categories $\fib_u$ as categories of representations of rational VOAs. 

Note that the formula 
\begin{equation}\label{cch}
u^{-1} = \xi(\fib_u)^2 = e^{\frac{\pi ic_V}{2}}
\end{equation}
allows us to determine the structure of a modular category for the category of representations of a Fibonacci VOA. The relation $\xi(\fib_u)^4 = u^{-2} = \theta_X$ implies that for a Fibonacci VOA $\V$ 
$$\frac{c_\V}{2}\equiv h\ \ (1),$$ where $h$ is the conformal weight of the non-trivial irreducible representation of $\V$. 

\begin{example}({Affine VOA} $G_{2,1}$)

The central charge of $G_{2,1}$ is $14/5$. 
The conformal weights of irreducible $G_{2,1}$-modules are $0$ and $2/5$. 
\newline
Since the fusion rule of  $G_{2,1}$ is of Fibonacci type, the value of the central charge (or conformal weights) for  $G_{2,1}$ implies that
$$\Rep(G_{2,1})\cong \fib_{e^{\frac{3\pi i}{5}}}.$$
The conformal embedding $A_{1,28}\subset G_{2,1}$ allows us to write two irreducible $G_{2,1}$-modules as sums of irreducible $A_{1,28}$-modules:
$$L_{G_{2,1}}(0) = L_{A_{1,28}}(0)\oplus L_{A_{1,28}}(10)\oplus L_{A_{1,28}}(18)\oplus L_{A_{1,28}}(28),$$ $$L_{G_{2,1}}(1) = L_{A_{1,28}}(6)\oplus L_{A_{1,28}}(12)\oplus L_{A_{1,28}}(16)\oplus L_{A_{1,28}}(22).$$

Another conformal embedding $A_{1,3}\times A_{1,1}\subset G_{2,1}$ allows us to identify $G_{2,1}$ with the simple current extension of $A_{1,3}\times A_{1,1}$. 
Indeed, the $A_{1,3}\times A_{1,1}$-module $L_{A_{1,3}}(3)\boxtimes L_{A_{1,1}}(1)$ is invertible (simple current) and has conformal weights 1. 
The commutative algebra 
$$A = L_{A_{1,3}}(0)\boxtimes L_{A_{1,1}}(0)\oplus L_{A_{1,3}}(3)\boxtimes L_{A_{1,1}}(1)$$ 
(the simple current extension) in $\Rep(A_{1,3}\times A_{1,1}) = \Rep(A_{1,3})\boxtimes \Rep(A_{1,1})$ coincides with $G_{2,1}$. 
In particular, the non-trivial irreducible local $A$-module (the irreducible $\V$-module) has the following decomposition into irreducible $A_{1,3}\times A_{1,1}$-modules:
$$L_{A_{1,3}}(2)\boxtimes L_{A_{1,1}}(0)\oplus L_{A_{1,3}}(1)\boxtimes L_{A_{1,1}}(1).$$

\end{example}

\begin{example}({Affine VOA} $F_{4,1}$)

The central charge of $F_{4,1}$ is $26/5$. 
The conformal weights of irreducible $F_{4,1}$-modules are $0$ and $3/5$. 
\newline
The fusion rule of  $F_{4,1}$ is of Fibonacci type. Thus by the formula (\ref{cch}) 
$$\Rep(F_{4,1})\cong \fib_{e^{\frac{7\pi i}{5}}}.$$
\newline
The conformal embedding $A_{2,2}\times A_{2,1}\subset F_{4,1}$ allows us to identify $F_{4,1}$ with the simple current extension of $A_{2,2}\times A_{2,1}$.
\newline
The conformal embedding $G_{2,1}\times F_{4,1}\subset E_{8,1}$ allows us to identify $F_{4,1}$ with the centraliser (coset) of $G_{2,1}$ in $E_{8,1}$:
$$F_{4,1} = \frac{E_{8,1}}{G_{2,1}}.$$ 
Note that the converse is also true: $G_{2,1}$ is the centraliser (coset) of $F_{4,1}$ in $E_{8,1}$.

\end{example}

\begin{example}({Simple current extension of} $A_{1,8}$)

The irreducible $A_{1,8}$-module $L_{A_{1,8}}(8)$ is invertible (simple current) and has conformal weight 1. 
The commutative algebra $A = L_{A_{1,8}}(0)\oplus L_{A_{1,8}}(8)$ (the simple current extension) is maximal in $\Rep(A_{1,8})$ and $\Rep(A_{1,8})_A^{loc}$ has the type $\fib^{\boxtimes 2}$. Thus $\V = L_{A_{1,8}}(0)\oplus L_{A_{1,8}}(8)$ has a structure of VOA (a VOA extension of $A_{1,8}$) such that $\Rep(\V)$ is equivalent to $\fib^{\boxtimes 2}$ as a monoidal category. 

Indeed it follows from the $A_{1,8}$-fusion rule that there are five non-isomorphic induced $A$-modules 
$$A\otimes L_{A_{1,8}}(0) = A\otimes L_{A_{1,8}}(8),\quad A\otimes L_{A_{1,8}}(1) = A\otimes L_{A_{1,8}}(7),$$ $$A\otimes L_{A_{1,8}}(2) = A\otimes L_{A_{1,8}}(6),\quad A\otimes L_{A_{1,8}}(3) = A\otimes L_{A_{1,8}}(5),\quad A\otimes L_{A_{1,8}}(4)$$
and all but the last one are irreducible as $A$-modules. 
The last one is a sum of two (non-isomorphic) $A$-modules.
By looking at conformal weights it can be seen that the $A$-modules induced from $L_{A_{1,8}}(1)$ and $L_{A_{1,8}}(3)$ are non-local and the rest is local.
Thus irreducible local $A$-modules (the irreducible $\V$-modules) have the following decompositions into irreducible $A_{1,8}$-modules:
$$L_{A_{1,8}}(0)\oplus L_{A_{1,8}}(8),\quad L_{A_{1,8}}(2)\oplus L_{A_{1,8}}(6),\quad L_{A_{1,8}}(4),\quad L_{A_{1,8}}(4).$$

The central charge of $\V$ coincides with the central charge of $A_{1,8}$ and is equal to $12/5$. 
The conformal weights of irreducible $\V$-modules are $0$, $1/5$ and $3/5$ (the last one appearing twice). 
This in particular shows that although 
$$\Rep(V)\cong \fib_{e^{\frac{7\pi i}{5}}}\boxtimes\fib_{e^{\frac{7\pi i}{5}}}.$$
$\V$ is not a tensor product of two VOAs of type $\fib$ (in that case the conformal weight of the second irreducible module would be twice the conformal weight of the last two irreducibles). 

Conformal embeddings allow us to represent $\V$ as cosets. The conformal embedding $A_{1,8}\times G_{2,1}\times G_{2,1}\subset E_{8,1}$ factors through 
$$A_{1,8}\times G_{2,1}\times G_{2,1}\subset\V\times G_{2,1}\times G_{2,1}\subset E_{8,1}$$
and gives the following coset presentation for $\V$:
$$\V = \frac{E_{8,1}}{G_{2,1}\times G_{2,1}}.$$

The conformal embedding $A_{1,8}\times G_{2,1}\subset F_{4,1}$ factors through 
$$A_{1,8}\times G_{2,1}\subset\V\times G_{2,1}\subset F_{4,1}$$ 
and gives another coset presentation for $\V$:
$$\V = \frac{F_{4,1}}{G_{2,1}}.$$
\end{example}

Note that although the chiral algebras
$\frac{F_{4,1}}{G_{2,1}}\times E_{8,1}$ and $F_{4,1}\times F_{4,1}$  (here $E_{8,1}$ is the holomorphic chiral algebra of central charge 8) have equivalent categories of representations and the same central charge they are not isomorphic. The way to see it is to compare their degree 1 components.

\begin{example}({Minimal VOA} $M(2,5) = \Vir_{-\frac{22}{5}}$)

The minimal VOA with the central charge $c_{2,5} = -22/5$.
The irreducible modules $L_{2,5}(1,1)$ and $L_{2,5}(1,2)$ have conformal weights $0$ and $-1/5$. 

The fusion rule of  $M(2,5)$ is of Fibonacci type. Thus
$$\Rep(M(2,5))\cong \fib_{e^{\frac{\pi i}{5}}}.$$

Being minimal and maximal at the same time $M(2,5)=\Vir_{-\frac{22}{5}}$ is the only VOA of central charge $-22/5$ (which contains $\Vir_{-\frac{22}{5}}$).
\end{example}

\begin{example}({Simple current extension of minimal VOA} $M(3,10) = \Vir_{-\frac{44}{5}}$)

The minimal VOA $M(3,10)$ has central charge $c_{3,10} = -44/5$. Moreover we have an embedding
$$\Vir_{-\frac{44}{5}}\subset \Vir_{-\frac{22}{5}}\otimes \Vir_{-\frac{22}{5}}$$
Note that $\Vir_{-\frac{22}{5}}\otimes \Vir_{-\frac{22}{5}}$ is a simple current extension of $\Vir_{-\frac{44}{5}}$. Indeed its decomposition as an $\Vir_{-\frac{44}{5}}$-module is
$L_{3,10}(1,1)\oplus L_{3,10}(2,1).$
\end{example}

\begin{example}{Maximal extension of $M(3,5)^{\times 8}\times M(2,5)^{\times 7}$}

The central charge of the minimal model $M(3,5)$ is $-3/5$. The irreducible representations are 
$$1=L_{3,5}(1,1), \quad x=L_{3,5}(2,1),\quad y=L_{3,5}(1,2),\quad z=L_{3,5}(2,2).$$
Their conformal weights are 
$$h_1=0,\quad h_x=\frac{3}{4},\quad h_y=-\frac{1}{20},\quad h_z=\frac{1}{5}.$$
Fusion rules of $M(3,5)$ have the form:
$$
\begin{array}{c|cccc} 
\times & x &y & z \\  \hline  x & 1    \\ y & z & 1+z  \\ z & y & x+y & 1+z \end{array}$$

Thus the category $Rep(M(3,5))$ is a product of a Fibonacci category and a pointed category with the $\Z/2\Z$-fusion rule. The values of conformal weights of representations of $M(3,5)$ imply that
$$Rep(M(3,5)) \cong \fib_{e^{\frac{9\pi i}{5}}}\boxtimes \cC(\Z/2\Z,q,\theta),$$ with $\theta(1) = -i$. In particular the product $M(2,5)\times M(3,5)$ has an extension 
$$\V = L_{2,5}(1,1)\boxtimes L_{3,5}(1,1)\oplus L_{2,5}(1,2)\boxtimes L_{3,5}(2,2).$$ By looking at the characters it can be seen that although $\V$ is non-negatively graded $\V = \oplus_{n\geq 0}\V_n$, the degree zero component $\V_0$ is not one-dimensional. 
The representation category of $\V$ is $Rep(\V) = \cC(\Z/2\Z,q,\theta)$. The 8-th power $\cC(\Z/2\Z,q,\theta)$ has a trivialising algebra. This implies that the 8-th power $\V^{\times 8}$ has a holomorphic extension (of central charge $-40$). Similarly the category $\cC = Rep(M(3,5)^{\times 8}\times M(2,5)^{\times 7})$ has a maximal algebra $A$ such that the category of local modules $\cC_A^{loc}$ is $\fib_{e^{\frac{9\pi i}{5}}}$. The corresponding maximal extension $\U$ of $M(3,5)^{\times 8}\times M(2,5)^{\times 7}$ is a non-negatively graded VOA (although with $dim(\U_0)>1$) of central charge $-178/5$.

\end{example}

\appendix

\end{document}